\documentclass[twoside,12pt]{amsart}
\usepackage[a4paper,scale=.795,centering]{geometry}  
\usepackage{mathrsfs,amssymb,mathabx}
\usepackage[mathlines]{lineno}
\usepackage[dvipsnames]{xcolor}
\usepackage[pagebackref,colorlinks,citecolor=Plum,urlcolor=Periwinkle,linkcolor=DarkOrchid]{hyperref}
\usepackage{graphicx}
\newenvironment{lesn}{\begin{linenomath}\begin{esn}}{\end{esn}\end{linenomath}}
\usepackage{subfig}
 \theoremstyle{plain}
\newtheorem{theorem}{Theorem}

\newtheorem{lemma}[theorem]{Lemma}
\newtheorem{corollary}[theorem]{Corollary}
\newtheorem{proposition}[theorem]{Proposition}

\theoremstyle{definition}
\newtheorem*{definition}{Definition}
\newtheorem*{remark}{Remark}
\newcommand{\excspace}{\ensuremath{E}}
\newcommand{\dunderline}[1]{\underline{\underline{#1}}}
\newcommand{\rkmes}{\Xi}
\title[Splitting trees II]{Totally ordered measured trees and splitting trees with infinite variation II: 
\\ Prolific skeleton decomposition}
\author{Amaury Lambert}
\address{Laboratoire de Probabilit\'es, Statistique et Mod\'elisation (LPSM)\\
Sorbonne Universit\'e, CNRS, Paris, France\\
Case courrier 188\\
4, Place Jussieu\\
75252 PARIS Cedex 05}
\author{Ger\'onimo Uribe Bravo}
\address{Instituto de Matem\'aticas\\ 
Universidad Nacional Aut\'onoma de M\'exico\\
\'Area de la Investigaci\'on Cient\'ifica, Circuito Exterior, Ciudad Universitaria\\ Coyoac\'an, 04510. Ciudad de M\'exico, M\'exico }

\newcommand{\cadlag}{c\`adl\`ag}
\newcommand{\defin}[1]{{\bf #1}}
\newcommand{\sko}{\ensuremath{\mathbf{D}}}

\newcommand{\re}{\ensuremath{\mathbb{R}}}
\newcommand{\paren}[1]{\ensuremath{\left( #1\right) }}
\newcommand{\F}{\ensuremath{\mathscr{F}}}

\newcommand{\p}{\mathbb{P}}
\newenvironment{esn}{\begin{equation*}}{\end{equation*}}
\newcommand{\imf}[2]{\ensuremath{#1\!\paren{#2}}}
\newcommand{\se}{\ensuremath{\mathbb{E}}}

\DeclareMathOperator{\cb}{CB}
\DeclareMathOperator{\cbi}{CBI}
\newcommand{\bra}[1]{\ensuremath{\left[ #1\right] }}
\newcommand{\z}{\ensuremath{\mathbb{Z}}}
\newcommand{\na}{\ensuremath{\mathbb{N}}}
\newcommand{\set}[1]{\ensuremath{\left\{ #1\right\} }}
\newcommand{\indi}[1]{\si_{#1}}
\newcommand{\si}{{\ensuremath{\bf{1}}}}
\newcommand{\esp}[1]{\ensuremath{\se\! \left( #1 \right)}}
\newcommand{\realtree}{\ensuremath{\re}\nbd tree}
\newcommand{\bb}[1]{\mathbb{#1}}

\newcommand{\nbd}{\nobreakdash -}
\newcommand{\leb}{\text{Leb}}

\newcommand{\imi}[2]{#2^{-1}\!\paren{#1}}

\newcommand{\q}{\ensuremath{ \bb{Q}  } }
\newcommand{\sag}[1]{\sigma\!\paren{#1}}

\newcommand{\proba}[1]{\ensuremath{\sip\! \left( #1 \right)}}
\newcommand{\sip}{\mathbb{P}}
\newcommand{\mc}[1]{\ensuremath{\mathscr{#1}}}
\newcommand{\G}{\ensuremath{\mc{G}}}

\newcommand{\tr}[1]{\ensuremath{\mathbf{#1}}}

\renewcommand{\d}[1]{\ensuremath{\operatorname{d}\!{#1}}}
\newcommand{\eps}{\ensuremath{ \varepsilon}}
\newcommand{\fun}[3]{\ensuremath{#1:#2\to #3}}
\newcommand{\fund}[3]{\ensuremath{#1:#2\mapsto #3}}

\thanks{GUB's research is supported by
CoNaCyT
grant FC-2016-1946 and UNAM-DGAPA-PAPIIT grant IN115217.  
AL thanks the \emph{Center for Interdisciplinary Research in Biology} 
(Coll\`ege de France, Paris) for funding.} 
\usepackage[dvipsnames]{xcolor}

\begin{document}
\begin{abstract}
The first part of this paper (\cite{splitting1})  introduced splitting trees, 
those chronological trees admitting the self-similarity property where individuals give birth, 
at constant rate, to iid copies of themselves. 
It also established the intimate relationship between splitting trees and L\'evy processes. 
The chronological trees involved were formalized as Totally Ordered Measured (TOM) trees.

The aim of this paper is to continue this line of research in two directions: 
we first decompose locally compact TOM trees in terms of their prolific skeleton 
(consisting of its infinite lines of descent). 
When applied to splitting trees, 
this implies the construction of the supercritical ones (which are locally compact) 
in terms of the subcritical ones (which are compact) 
grafted onto a Yule tree (which corresponds to the prolific skeleton). 

As a second (related) direction, we study the genealogical tree associated to our chronological construction. 
This is done through the technology of the height process introduced by \cite{MR1954248}. 
In particular we prove a Ray-Knight type theorem 
which extends the one for (sub)critical L\'evy trees to the supercritical case.

%
\end{abstract}
\maketitle
\section{Introduction}
\label{introductionSection}
\subsection{Motivation}
\label{motivationSubsection}
One of the main results in this work is a Ray-Knight theorem for supercritical L\'evy trees 
which features a two-type branching process with values in $\na\times [0,\infty]$. 
Since (sub)critical L\'evy trees have been shown to correspond 
to scaling limits of (sub)critical Galton-Watson trees (as in Chapter 2 of \cite{MR1954248}), 
our Ray-Knight theorem is the continuous counterpart to the discrete result describing 
a supercritical Galton-Watson process in terms of a two-type Galton-Watson process, 
which is described in Chapter 5\S 7 of \cite{LP:book} as we now briefly recall. 

A plane tree is a combinatorial tree on which a total order has been defined, 
so that one is able to make sense of the first, second, etc offspring of each internal node. 
Using the Ulam-Harris-Neveu labelling, 
they are typically realized as subsets of the set $\mc{U}$ of words 
$u=u_1\cdots u_n$ on $\z_+=\set{1,2,\ldots}$, 
whose length $n$ is denoted $|u|$. 
The empty word is denoted $\emptyset$ and its length is zero. 
The concatenation of two words $u=u_1\cdots u_m$ and $v=v_1\cdots v_n$ 
is the word $uv$ given by $u_1\cdots u_mv_1\cdots v_m$ and with this concept, 
one can interpret the word $uj$ as the $j$-th child of $u$. 
We can then link this notion of tree with the genealogical one 
by declaring that $u$ precedes $v$ in the genealogical order, denoted $u\preceq v$, 
if there is a word $w$ such that $uw=v$. 
The set of labels $\mc{U}$ will also be equipped with the lexicographic total order. 

\begin{definition}
A \defin{plane tree} is a subset $\tau$ of $\mc{U}$ such that
\begin{enumerate}
\item $\emptyset\in\tau$ 
\item if $u=u_1\cdots u_{n+1}\in\tau\setminus\set{\emptyset}$ 
then the \defin{mother} of $u$, defined as $\imf{\pi}{u}=u_1,\ldots, u_n$, 
also belongs to $\tau$. 
\item If $u\in \tau$, there exists $\imf{k_u}{\tau}\in\na$ 
(interpreted as the quantity of descendants of $ u\in\tau$) 
such that $uj\in\tau$ if and only if $1\leq j\leq \imf{k_u}{\tau}$. 
\end{enumerate}
\end{definition}We can also define the rank of $u=u_1\cdots u_n$ as a sibling, denoted $\imf{r}{u}$, as $u_n$. 
The \defin{Lukasiewicz path} associated to a plane tree $\tau$ is the sequence $e$ 
obtained by first ordering the elements of the tree as $\emptyset=u_0<u_1<\cdots< u_{p-1}$, 
where $p=\#\tau$, and then defining  $e_0=0$ and $e_i-e_{i-1}=\imf{k_{u_i}}{\tau}-1$. 
It is characterized by being a skip-free excursion-like path: 
it starts at zero, its increments belong to $\set{-1,0,1,2,\ldots}$ 
and it remains non-negative until its last step, where it reaches $-1$. 
It is well known that finite plane trees are in bijection with the set of Lukasiewicz paths; 
we shall recall how to replicate this result in the continuum setting in 
Subsubsection \ref{treeCodedSubSubSection}. 

Galton-Watson trees were defined in \cite{MR850756} as the random plane tree $\Theta$ where: 
for a given offspring distribution $p=\paren{p_k,k\geq 1}$, 
let $\paren{\chi_u,u\in\mc{U}}$ be iid random variables with law $p$ 
and define the random tree $\Theta$ recursively constructed by
\begin{enumerate}
\item $\emptyset\in\Theta$. 
\item If $u\in\Theta$ then $uj\in\Theta$ if and only if $1\leq j\leq \chi_u$. 
\end{enumerate}
Hence, for every $u\in\Theta$, the random quantity $\chi_u$ 
is  the quantity of descendants of $u$. 

Define the $n$-th generation of $\Theta$ as $\mc{G}_n=\set{u\in\Theta: |u|=n}$; 
its size will be denoted $Z_n$. 
Then, $Z$ is a Galton-Watson process with offspring distribution $p$. 
From the extinction criteria for the latter, 
it follows that $\Theta$ is finite with probability $1$ if and only if $p$ is (sub)critical: 
$\sum_k k p_k\leq 1$. 

The following definition is therefore relevant only in the supercritical case: 
\begin{definition}
Let $\tau$ be a plane tree and let $u\in\tau$. 
The \defin{subtree of $\tau$ above $u$} is the plane tree $\tau_u$ 
consisting of words $v$ such that $uv\in\tau$. 
An element $u$ of $\tau$ is called a \defin{prolific individual} of $\tau$ if $\tau_u$ is infinite. 
\end{definition}The set of prolific individuals of $\tau$ will be denoted $\mc{P}_\tau$. 
In Figure \ref{superCriticalGWTree} one has a plane tree on which prolific individuals have been identified. 

\begin{figure}
\includegraphics[width=.4\textwidth]{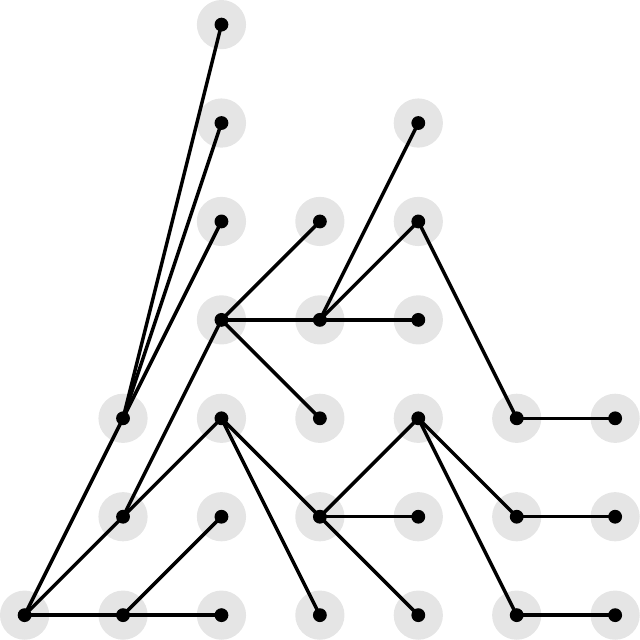}
\hfill
\includegraphics[width=.4\textwidth]{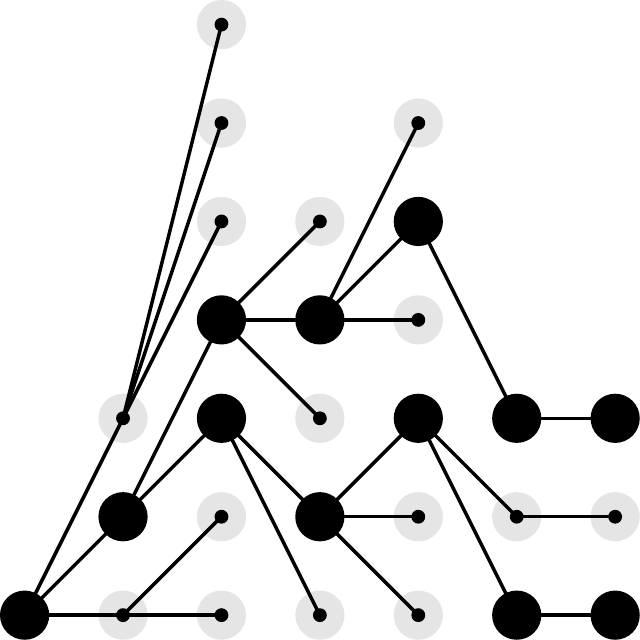}
\caption{Left: the first 7 generations of an infinite plane tree. 
Generations increase from left to right. 
On each generation, 
labels (lexicographically) increase from bottom to top. 
(Hence, the tree is $\set{\emptyset,1,2,3,11,12,21,22,31,32,33,\ldots}$). 
Right: the prolific individuals are identified by black disks. 
Notice that the root has three subtrees above it: two finite ones and an infinite one.}
\label{superCriticalGWTree}
\end{figure}

For our Galton-Watson tree $\Theta$, consider
\begin{linenomath}
\begin{esn}
Z^1_n=\# \mc{G}_n\cap \mc{P}_\Theta
\quad\text{and}\quad
Z^2_n=Z_n-Z^1_n.
\end{esn}
\end{linenomath}Then, the two-type branching process alluded to above is $\paren{Z^1,Z^2}$. 

One of our objectives will be to develop models of random real trees 
on which the above decomposition into two-type branching processes can be carried out. 

\subsection{Preliminaries}
In this short subsection, we recall the setting and results of the prequel \cite{splitting1} that we will use. 

\subsubsection{Spectrally positive L\'evy processes}
We will mainly concentrate on spectrally positive L\'evy processes. 
An adequate background is found in \cite{MR1406564}, especially Chapter VII. 
We use the canonical setup. 
There will be two canonical spaces: 
the Skorohod space $\sko$ of \cadlag\ functions $\fun{f}{[0,\infty)}{\re\cup\set{\dagger}}$  
and the (positive) excursion space 
$\excspace$ consisting of \cadlag\ functions $\fun{f}{[0,\infty)}{[0,\infty)\cup\set{\dagger}}$ 
for which there exists a lifetime $\zeta=\imf{\zeta}{f}\in [0,\infty]$ 
such that $f>0$ on $(0,\zeta)$ and $f=\dagger$ after $\zeta$. 
(As usual, $\dagger$ stands for an isolated cemetery state.) 
We recall that on both spaces, the canonical process $X$ can be defined by 
$\imf{X_t}{f}=\imf{f}{t}$ and equipped with the canonical filtration $\F_t=\sag{X_s: s\leq t}$. 

Let $\Psi$ be the Laplace exponent of a possibly killed spectrally positive L\'evy process. 
The function $\Psi$ is characterized in terms of the L\'evy quartet $\paren{\kappa, \alpha, \beta,\pi}$
where $\kappa\geq 0$, $\alpha\in\re$, $\beta\geq 0$ 
and $\pi$ is a measure on $(0,\infty)$ satisfying $\int \bra{1\wedge x^2}\, \imf{\pi}{dx}<\infty$. 
The characterization is expressed through the L\'evy-Kintchine formula as follows: 
\begin{linenomath}
\begin{esn}
\imf{\Psi}{\lambda}=
-\kappa +\alpha \lambda+\beta\lambda^2+
\int_0^\infty \bra{e^{-\lambda x}-1+\lambda x\indi{x\leq 1}}\, \imf{\pi}{dx}. 
\end{esn}
\end{linenomath}

Recall that $\Psi$ gives rise to a (sub)Markovian family of probability laws on $\sko$, 
say $\paren{\p_x,x\in\re}$, 
such that each $\p_x$ is (sub)Markovian 
and they are spatially homogeneous (the image of $\p_x$ under the mapping $f\mapsto y+f$ is $\p_{x+y}$). 
The link between $\p_x$ and $\Psi$ is: 
\begin{linenomath}
\begin{esn}
\imf{\se_x}{e^{-\lambda X_t}}=e^{t\imf{\Psi}{\lambda}}. 
\end{esn}\end{linenomath}We assume that $\Psi$ does not correspond to a subordinator, 
which is equivalent to saying that 
$\imf{\Psi}{\lambda}\to \infty$ as $\lambda\to\infty$. 
Since $\Psi$ is convex, $\Psi$ has at  most two roots. 
We let $b$ stand for the biggest root of $\Psi$ 
and define the associated Laplace exponent $\Psi^\#$ 
defined by $\imf{\Psi^\#}{\lambda}=\imf{\Psi}{\lambda+b}$. 
The Laplace exponent $\Psi^\#$ can be obtained by conditioning $\p_x$ on reaching arbitrarily low levels 
as in Lemma 7 in \cite[Ch. 7]{MR1406564} and Lemme 1 in \cite{MR1141246}. 
We say that $\Psi$ is supercritical if $b>0$ and (sub)critical otherwise.

Let $\Psi$ be a supercritical Laplace exponent 
and $\p_x$ the law of a spectrally positive L\'evy process with Laplace exponent $\Psi$ started at $x$. 
Since $X$ drifts to $\infty$ under $\p_0$, the minimum $\underline X_\infty$ of $X$ belongs to $(-\infty,0]$. 
We can then define $T_m$ as the last time the minimum of $X$ is approached as a left limit; 
note that $X$ might have a positive jump at $T_m$.
The post-minimum process $X^{\rightarrow}$ is defined as
\begin{linenomath}\begin{esn}
X^{\rightarrow }_{t}=X_{T_m+t}-X_{T_m-}.
\end{esn}\end{linenomath}Note that $X^\rightarrow$ does not start at zero if $X$ jumps at $T_m$. 
The law of this process is $\p^{\rightarrow}$. 
It has the important properties of being Markovian, 
that $X^\rightarrow_t>0$ for $t>0$ , 
and for any $t>0$, conditionally on $X_t=x>0$, 
the shifted process $X_{t+\cdot}$ has law $\p_x$ conditioned on not reaching zero. 
Later, we will assume Grey's hypothesis on $\Psi$, 
which in particular implies that the process is of infinite variation. 
In terms of $\Psi$, we will have either $\sigma>0$ 
or $\int [1\wedge x]\, \imf{\pi}{dx}$ 
and then $X$ reaches its minimum at a unique place and continuously 
(cf. Proposition 2.1 in \cite{MR0433606} or Proposition 1 of \cite{MR2978134}).


\subsubsection{The compact tree coded by a \cadlag\ function}
\label{treeCodedSubSubSection}
Let $\fun{f}{[0,m]}{[0,\infty)}$ be a \cadlag\ function such that $\imf{f}{m}=0$. 
The \defin{tree coded by $f$} is defined as follows. 
For $s,t\in [0,m]$, define
\begin{linenomath}
\begin{esn}
\imf{d_f}{s,t}=\imf{f}{s}+\imf{f}{t}-2\imf{m_f}{s,t}
\quad\text{where}\quad
\imf{m_f}{s,t}=\inf_{r\in [s,t]}\imf{f}{r}.
\end{esn}
\end{linenomath}Then $d_f$ is a pseudometric on $[0,m]$ and we define $\tau_f$ as 
the set of equivalence classes $[t]_f$ induced by 
the associated equivalence relationship $\sim_f$ where $s\sim_f t$ if $\imf{d_f}{s,t}=0$. 
The induced metric by $d_f$ on $\tau_f$ turns 
the space $\paren{\tau_f,d_f}$ into a compact real tree, 
whose definition we now recall. 
\begin{definition}[From \cite{MR1397084} and \cite{MR2221786}]
An \defin{\realtree}\ (or \defin{real tree}) is a metric space $\paren{\tau,d}$ satisfying the following properties:
\begin{description}
\item[Completeness] $\paren{\tau,d}$ is complete.
\item[Uniqueness of geodesics]  For all $\sigma_1,\sigma_2\in\tau$ 
there exists a unique isometric embedding\begin{linenomath}\begin{esn}
\fun{\phi_{\sigma_1,\sigma_2}}{[0,\imf{d}{\sigma_1,\sigma_2}]}{\tau}
\end{esn}\end{linenomath}such that $\imf{\phi}{0}=\sigma_1$ 
and $\imf{\phi}{\imf{d}{\sigma_1,\sigma_2}}=\sigma_2$.
\item[Lack of loops] For every injective continuous mapping $\fun{\phi}{[0,1]}{\tau}$ 
such that $\imf{\phi}{0}=\sigma_1$ 
and $\imf{\phi}{1}=\sigma_2$, the image of $[0,1]$ under $\phi$ equals 
the image of $[0,\imf{d}{\sigma_1,\sigma_2}]$ under $\phi_{\sigma_1,\sigma_2}$. 
\end{description}
A triple $\paren{\tau,d,\rho}$ consisting of a real tree $\paren{\tau,d}$ 
and a distinguished element $\rho\in\tau$ is called a \defin{rooted (real) tree}.
\end{definition}

On a rooted real tree $\paren{\tau,d,\rho}$, 
the image of $[0,\imf{d}{\sigma_1,\sigma_2}]$ under 
the isometry $\phi_{\sigma_1,\sigma_2}$ is denoted $[\sigma_1,\sigma_2]$. 
We also define the associated \emph{genealogical partial order} $\preceq$, 
where $\sigma_1\preceq \sigma_2$ if $\sigma_1\in [\rho,\sigma_2]$. 
On $\tau_f$, where the root is chosen as $[m]_f$, 
if $\sigma_i=[t_i]_f$, then $\sigma_1\preceq\sigma_2$ if and only if $\imf{m_f}{t_1,t_2}=\imf{f}{t_1}$. 

The main proposal of \cite{splitting1}, 
adapted from \cite{Duquesne:fk}, 
is to endow the compact rooted real tree $\paren{\tau_f,d_f,\rho}$ 
 with additional structure inherited from $[0,m]$: 
a total order $\leq$ where $[s]_f\leq [t]_f$ if $\sup[s]_f\leq \sup[t]_f$, 
and the measure $\mu$ given by the image of $\leb$ under the projection $t\mapsto [t]_f$. 
The triplet $\paren{\paren{\tau_f,d_f,\rho},\leq ,\mu}$ 
constitutes a compact Totally Ordered Measured (TOM) tree. 
\begin{definition}
A \defin{real tree} $\paren{\tau,d,\rho}$ is called 
\defin{totally ordered} if there exists a total order $\leq$ on $\tau$ which satisfies
\begin{description}
\item[Or1] $\sigma_1\preceq \sigma_2$ implies $\sigma_2\leq \sigma_1$ and
\item[Or2] $\sigma_1<\sigma_2$ implies $[\sigma_1,\sigma_1\wedge\sigma_2)<\sigma_2$.
\end{description}A totally ordered real tree is called \defin{measured} if there exists a 
measure $\mu$ on the Borel sets of $\tau$ satisfying:
\begin{description}
\item[Mes1] 
$\mu$ is locally finite and for every $\sigma_1<\sigma_2$:
\begin{linenomath}
\begin{esn}
\imf{\mu}{\set{\sigma:\sigma_1<\sigma\leq \sigma_2}}>0. 
\end{esn}
\end{linenomath}
\item[Mes2] $\mu$ is diffuse. 
\end{description}A totally ordered measured tree will be referred to as a \defin{TOM tree}. 
\end{definition}

The importance of this notion is that compact TOM trees 
are precisely those that can be coded by a function in a canonical manner
(Cf. Theorem 1 of \cite{splitting1}, 
adapted from Theorem 1.1 in \cite{Duquesne:fk}). 
In a sense, we replicate the concept of plane trees 
(a setting which has proved very useful for Galton-Watson processes)
in the continuous setting thanks to the total order and the measure. 

We now define the random TOM trees that will interest us in the compact case. 
Let $\Psi$ be a (sub)critical exponent. 
Then $\liminf_{t\to\infty} X_t=-\infty$ under $\p_0$, 
so that the cumulative minimum process $\underline X$ given by 
$\underline X_t=\inf_{s\leq t} X_s$, satisfies $\underline X_\infty=-\infty$. 
Hence, $0$ is recurrent for the Markov process $X-\underline X$ 
and we can then define $\nu=\nu^\Psi$ as the excursion measure of $X-\underline X$ 
(cf. Chapter VI in \cite{MR1406564}) . 
We then consider the measure $\eta=\eta^\Psi$ 
equal to the image of $\nu$ under the map that sends excursions into TOM trees. 

\subsubsection{Locally compact trees and their coding sequence}
\label{codingSequenceSubSubSection}
Let us now recall how to obtain a locally compact TOM tree 
out of a sequence of functions \emph{compatible under pruning}.
Let $\paren{f_n,n\geq 0}$ be a sequence of \cadlag\ functions on $[0,m_n]$. 
We say that the sequence is compatible under pruning if, 
for every $n\geq 1$ there exists a set $B_n\subset[0,m_n]$ such that, 
on defining
\begin{linenomath}
\begin{esn}
\tilde B_n=\set{t\in [0,m_n]: \exists s\in B_n, [s]_{f_n}\preceq [t]_{f_n}}, 
\quad A^n_t=\imf{\leb}{[0,t]\setminus \tilde B^n}\quad\text{and}\quad C^n=(A^n)^{-1},
\end{esn}\end{linenomath}we have the equality
\begin{linenomath}\begin{esn}
f_{n-1}=f_n\circ C^n.
\end{esn}\end{linenomath}Heuristically, the set $B_n$ selects nodes on the tree coded by $f_n$ 
and the time-change $C^n$ removes whatever is on top of them. 
It follows that the compact TOM tree $\tr{c}_{n-1}$ coded by $f_{n-1}$ can be embedded into $\tr{c}_{n}$. 
Indeed, one can prove that the map 
$\fund{\phi_{n-1}}{t}{C^n_t}$ is constant on the equivalent class of $[t]_{f_{n-1}}$ 
and use this to construct the embedding. 
Under the condition\begin{linenomath}\begin{esn}
\lim_{n\to\infty}
\inf_{t\in B_n}\imf{f}{t}
=\infty
\end{esn}\end{linenomath}and reasoning as in the proof of Proposition 5 of \cite{splitting1}, 
we conclude the existence of a unique locally compact TOM tree $\tr{c}=\paren{\paren{\tau,d,\rho},\leq,\mu}$ 
such that each $\tr{c}_n$ can be embedded (in a growing manner) into $\tr{c}$ 
and such that the embeddings exhaust $\tr{c}$. 
More formally, we might define $\tilde\tau=\bigcup \set{i}\times \tau_i$ 
and then define $\tau$ as the quotient of $\tilde \tau$ under the 
equivalence relationship $(i,\sigma_i)\sim (j,\sigma_j)$ (where $i<j$, say) 
whenever $\imf{\phi_{j-1}\cdots \circ\phi_{i}}{\sigma_i}=\sigma_j$. 
(Other structural parts of $\tr{c}$ can be defined analogously.) 
Then, the embedding would just send $\sigma\in\tr{c}_i$ to $(i,\sigma)$. 
We say that $(f_n)$ is a \defin{coding sequence} for the locally compact TOM tree $\tr{c}$. 
If $\tr{\tilde c}$ is any other TOM tree with this property, 
then $\tr{c}$ can be embedded into $\tr{\tilde c}$. 
A particular case of the above construction is 
when the sequence of functions $\paren{f_n}$ are consistent under truncation at levels $r_n$, 
where the sequence $(r_n)$ is non-decreasing, 
in which the set $B_n$ consists of $t\in [0,m_n]$ such that $\imf{f_n}{t}> r_{n-1}$. 
In this case, $\tilde B_n=B_n$ and the time-change $C_n$ 
removes the set of $t$ such that $f_n(t)>r_{n-1}$ and closes up the gaps. 
We refer to this as time changing $f_n$ to remain below $r_{n-1}$.

\subsection{Statement of the results}
\label{resultsSubsection}

The locally compact TOM trees that will interest us 
come from the Laplace exponent of a supercritical Laplace exponent $\Psi$ as is now described. 
Recall that $b$ denotes the greatest root of $\Psi$ and that the associated Laplace exponent is given by $\Psi^\#$. 
The reflected process $X-\underline X$ under $\p_0$  is now transient, 
which in terms of its construction by excursions, 
means that its excursion measure charges those with infinite length. 
Let $\nu$ be the excursion measure of $X-\underline X$ under $\p$ and $\nu^{\#}$ the same excursion measure under $\p^\#$. 
Then, $\nu=\nu^\#+b\p^\rightarrow$. 
Let $\q^{\rightarrow, r}$ be the probability measure constructed by concatenating, to a process obtained by time-changing a process with law $\p^{\rightarrow}$ to remain below $r$, independent copies of a process with law $\p_r$ time-changed to remain below $r$, until the first copy reaches zero, followed by killing. 
We then define $\nu^r=\nu^{\#,r}+b\q^{\rightarrow ,r}$, where $\nu^{\#,r}$ is the image of $\nu^{\#}$ under time-change to remain below $r$. 

By construction, the measures $\nu^r$ are consistent under truncation, 
meaning that if $r_1\leq r_2$ then $\nu^{r_1}$ is the image of $\nu^{r_2}$ under time change to remain below $r_1$.
Hence,  a unique measure $\eta^\Psi$ on locally compact TOM trees can be defined 
so that $\nu^r$ equals the image of $\eta^\Psi$ under the function which takes a tree into the contour of its truncation at level $r$. 

\defin{Splitting trees} are those whose \emph{law} is $\eta^\Psi$, either in the (sub)critical or supercritical cases. 
They have been characterized as the $\sigma$-finite laws on locally compact TOM trees satisfying a certain self-similarity property termed the splitting property in Theorem 2 of \cite{splitting1}. 

In this work, we will be interested in analyzing the measure $\eta^\Psi$. 
We will first be concerned with the descriptions of the prolific individuals. 

A particular case of the  construction of $\eta^\Psi$ is the Yule tree. 
It is obtained with the L\'evy process $X_t=-t$ killed at rate $b$, for which $\imf{\Psi}{\lambda}=\lambda-b$. 
The interpretation is that individuals have infinite life-times 
(which correspond to interpreting killing as making an infinite jump) and that they give birth at rate $b$. 
The 
measure $\nu^{\#,r}$ is zero, while $\q^{\rightarrow, r}$ 
has a simple description: 
let $(T_n)$ be a Poisson point process on $[0,\infty)$ with intensity $b\, \leb$, set $S_i=T_i-T_{i-1}$  and $N_r=\min \set{i\geq 1: S_i>r}$.  
We let\begin{equation}
\label{heightProcessOfTruncatedYuleTreeDefinition}
X^r_t=\sum_{n=1}^{N_r}\bra{r-\paren{t-T_{n-1}}}\indi{T_{n-1}\leq t<T_{n}}
\end{equation}on the interval $[0, T_{N_r-1}+r]$. 
A simple consequence of this description of the Yule tree 
is that the quantity of individuals alive at time $r$, 
which evolve as the usual Yule process and correspond to the number of jumps of $X^r$ until reaching zero, 
has a geometric distribution of parameter $1-e^{-b r}$. 
This is a classical result which is usually proved using the Kolmogorov equations. 

As we shall see, Yule trees appear in supercritical splitting trees. 
Indeed, the latter can be obtained by first constructing a skeleton of infinite lines of descent, 
which is a Yule tree, 
and then grafting onto it supercritical splitting trees conditioned on extinction. 
The latter turn out to be a special kind of subcritical splitting tree. 
We first explore the notions of infinite lines of descent and of grafting. 

\begin{definition}
Let $\tr{c}=\paren{\paren{\tau,d,\rho}, \leq, \mu}$ be a locally compact TOM tree. 
An \defin{infinite line of descent} is an isometry $\fun{\phi}{[0,\infty)}{\tau}$ such that $t\mapsto \imf{d}{ \rho, \imf{\phi}{t}}$ is increasing. 
We say that $\sigma\in\tau$ has an infinite line of descent 
if $\sigma$ belongs to the image of an infinite line of descent. 
\end{definition}
We will now give a genealogical structure to the infinite lines of descent. 
\begin{proposition}
\label{backboneProposition}
Let $\mc{I}$ be the collection of individuals with infinite lines of descent. 
Then $\mc{I}=\emptyset$ if and only if $\tau$ is compact. 
If $\tau$ is non-compact, 
$\mc{I}$ is a non-compact connected subset of $\tau$ containing the root 
which can be given the structure of a locally compact TOM trees as follows: 
the tree structure (geodesics and lack of loops) is inherited from $\tau$, as is the total order, and there exists a naturally defined Lebesgue measure on $\mc{I}$ which assigns to any interval $[\rho,\sigma]$ its \emph{length} $\imf{d}{\rho,\sigma}$. 
Furthermore, there exists a plane tree $\tau_I\subset\mc{U}$ 
and a collection of infinite lines of descent $\paren{I_u:u\in\tau_I}$ of $\tau$, 
with images $(\mc{I}_u,u\in\tau_p)$,  which partition $\mc{I}$ as follows: 
\begin{enumerate}
\item  $\bigcup_{u\in\tau_I} \mc{I}_u=\mc{I}$ and
\item on defining  $\sigma_u=\imf{I_u}{0}$, we have $\mc{I}_u\cap \mc{I}_{\imf{\pi}{u}}=\set{\sigma_u}$ and $\mc{I}_u\cap \mc{I}_v=\emptyset$ if $u\neq \imf{\pi}{v}$ or $v\neq \imf{\pi}{u}$. 
\end{enumerate}%
Furthermore, if $\alpha_u=\imf{d}{\rho,\sigma_u}$ for $u\in\tau_I$, 
then $\mc{I}$ can be uniquely reconstructed from the marked plane tree $\paren{\tau_I,\alpha}$. 
\end{proposition}
Heuristically, the infinite lines of descent are formed out of the plane tree $\tau_I$, by stipulating that individuals $u\in\tau_I$ live an infinite amount of time, and their offspring $uj$ are born at time $\imf{d}{\sigma_{uj},\rho}$. 
In the case of the Yule tree, the birth times are the jump times of a Poisson process of rate $b$ along each infinite line of descent. 

Let us turn to the notion of grafting. 
Let $\tr{c}_i=\paren{\paren{\tau_i,d_i,\rho_i},\leq_i,\mu_i}$ be two locally  compact TOM trees 
and consider $\sigma\in \tau_1$. 
We wish to graft $\tr{c}_2$ to $\tr{c}_1$ at $\sigma$.  
\begin{definition}
The \defin{grafting} of $\tr{c}_2$ \defin{to the right} of $\sigma\in\tau_1$ is the locally compact tree 
$\tr{c}=\paren{\paren{\tau,d,\rho},\leq,\mu}$ 
defined as follows: let\begin{linenomath}\begin{esn}
\tau=\bigcup_{i=1}^2\set{i}\times \tau_i,
\end{esn}\end{linenomath}equipped with the distance $d$ given by\begin{linenomath}\begin{esn}
\imf{d}{\paren{i,\sigma_1},\paren{j,\sigma_2}}=\begin{cases}
\imf{d_i}{\sigma_1,\sigma_2}&i=j\\
\imf{d_1}{\sigma_1, \sigma}+\imf{d_2}{\rho_2,\sigma_2}& i=1,j=2
\end{cases}
\end{esn}\end{linenomath}and rooted at $(1,\rho_1)$. 
We now define a compatible order $\leq$ by stipulating that\begin{linenomath}\begin{esn}
\paren{i,\sigma_1}\leq \paren{j,\sigma_2}\quad\text{if and only if either }\begin{cases}
i=j\text{ and }\sigma_1\leq_i\sigma_2\\
i=1,j=2\text{ and }
\sigma_1\leq \sigma
\\
i=2,j=1\text{ and }
\sigma_2> \sigma
\end{cases}.
\end{esn}\end{linenomath}Finally, 
we extend $\mu_i$ to $\set{i}\times \tau_i$ in the obvious manner and, 
abusing notation, set $\mu=\mu_1+\mu_2$. 
\end{definition}
It can be seen that $\tr{c}=\paren{\paren{\tau,d,\rho},\leq,\mu}$ is a locally compact TOM tree.

If $f_i$ codes the compact tree $\tr{c_i}$, $\sigma=[t]_{f_1}$ 
and $t=\sup [t]_{f_1}$, 
then we can code $\tr{c}$ by the function $f$ given by\begin{linenomath}\begin{esn}
\imf{f}{s}=\begin{cases}
\imf{f_1}{s}&s< t\\
\imf{f_1}{t}+\imf{f_2}{s-t}&  t\leq s<t+\imf{\mu_2}{\tau_2}\\
\imf{f_1}{s-\imf{\mu_2}{\tau_2}}&t+\imf{\mu_2}{\tau_2}\leq s\leq \imf{\mu_1}{\tau_1}+\imf{\mu_2}{\tau_2}
\end{cases}. 
\end{esn}\end{linenomath}

One can give a more geometric construction of the Yule tree using grafting as follows. 
We start with $I_\emptyset=[0,\infty)$ (seen as a TOM tree). 
We next run a rate $b$ Poisson process along $I_\emptyset$ 
and at its jump times, 
we graft copies of $[0,\infty)$, say $I_1,I_2,\ldots$. 
The same procedure is then recursively repeated along each grafted copy. 
The tree so constructed, termed the Yule tree and denoted $I$, 
is the unique random locally compact TOM tree 
which has the same law as the tree obtained by grafting iid trees with the same law as $I$ 
on the interval $[0,\infty)$ at the jump times of an independent Poisson process.

The Yule tree is the simplest example of a locally compact splitting tree since all of its individuals live indefinitely. 
For more general locally compact splitting trees, 
we must accommodate individuals with finite and infinite lines of descent. 
However, the infinite lines of descent evolve analogously to Yule trees, 
on which compact trees are then grafted to the left and to the right. 
Locally compact trees with only one infinite line of descent above the root are called trees with a single infinite end, 
or \defin{sin trees}, following the terminology introduced in \cite{MR1102319}. 

We first  define the sin trees. 
Informally, the sin tree has left and right-hand sides: 
the left-hand side is coded by the post-minimum process of a L\'evy process 
with Laplace exponent $\Psi$ started at zero  
while the right-hand side is coded by a L\'evy process with Laplace exponent $\Psi^{\#}$ 
which \emph{starts at $\infty$} and is killed upon reaching zero. 
Formally, the sin tree is the unique random locally compact TOM tree 
whose truncation at level $r$ is the concatenation of the post-minimum process of a L\'evy process 
with Laplace exponent $\Psi$ started at zero 
and time-changed to remain below $r$ 
followed by a L\'evy process with Laplace exponent $\Psi^{\#}$ which starts at $r$, 
is time-changed to remain below $r$,  and is killed upon reaching zero. 
Its law will be denoted $\Upsilon$. 
It can be seen that under $\Upsilon$, 
there exists a unique infinite line of descent from the root almost surely (cf. Proposition \ref{uniquenessInfiniteLineOfDescentProposition}). 
Define the measure $\Upsilon_{\text{tree}}$ as the unique measure which equals the law of the grafting of iid copies of $\Upsilon_{\text{tree}}$ onto $\Upsilon$ 
along the unique infinite line of descent of the latter  at heights which correspond to the jump times 
of an independent Poisson process of intensity $b$. 
\begin{theorem}
\label{SupercriticalTreeFromGraftingTheorem}
Let $\Psi$ be a supercritical Laplace exponent and $b$ its largest root. 
The measure $\eta^\Psi$ on locally compact real trees can be described in terms of its restrictions to compact and non-compact trees as follows: 
\begin{linenomath}\begin{esn}
\eta^{\Psi}=\eta^{\Psi^{\#}}+b\Upsilon_{\text{tree}}.
\end{esn}\end{linenomath}
\end{theorem}
\begin{corollary}
\label{YuleTreeCorollary}
Under $\Upsilon_{\text{tree}}$, 
the TOM tree of infinite lines of descent has the same law as a Yule tree with birth rate $b$. 
\end{corollary}

We now pass to the description of the genealogical tree associated to supercritical splitting trees. 
As a motivation, consider the case where $\Psi$ is the Laplace exponent of 
a compound Poisson process with drift $-1$, as in (A) of Figure \ref{DiscreteChronologicalTree}. 
\begin{figure}
\begin{center}
\subfloat[][]{\includegraphics[width=.4\textwidth]{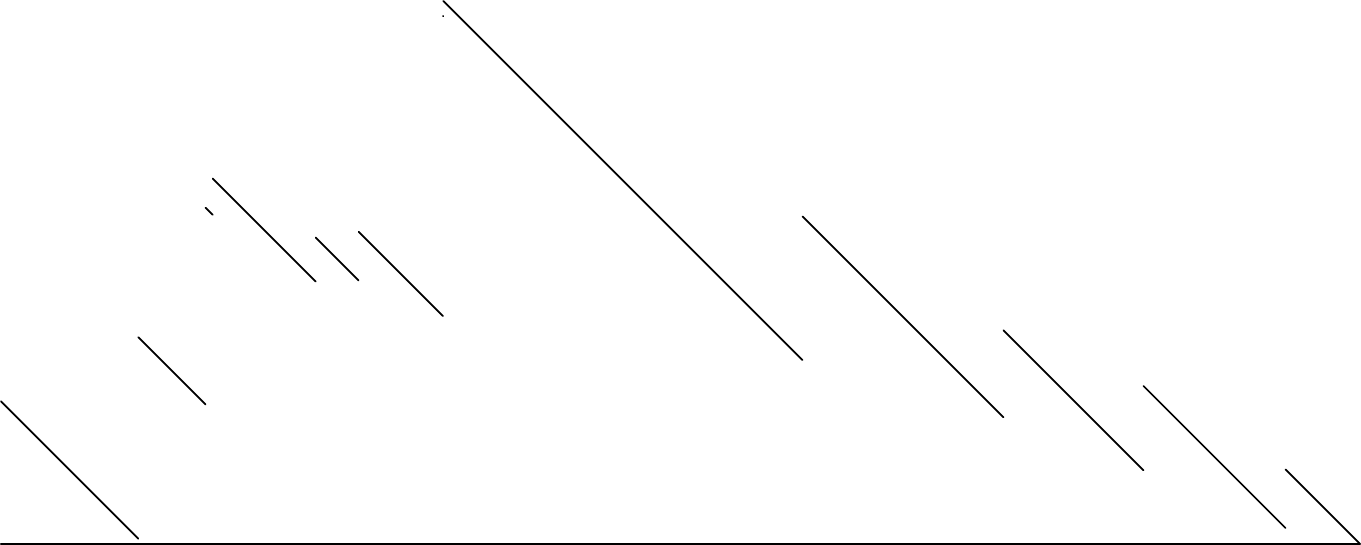}\label{codingProcessModel1}}
\hfill
\subfloat[][]{\includegraphics[width=.4\textwidth]{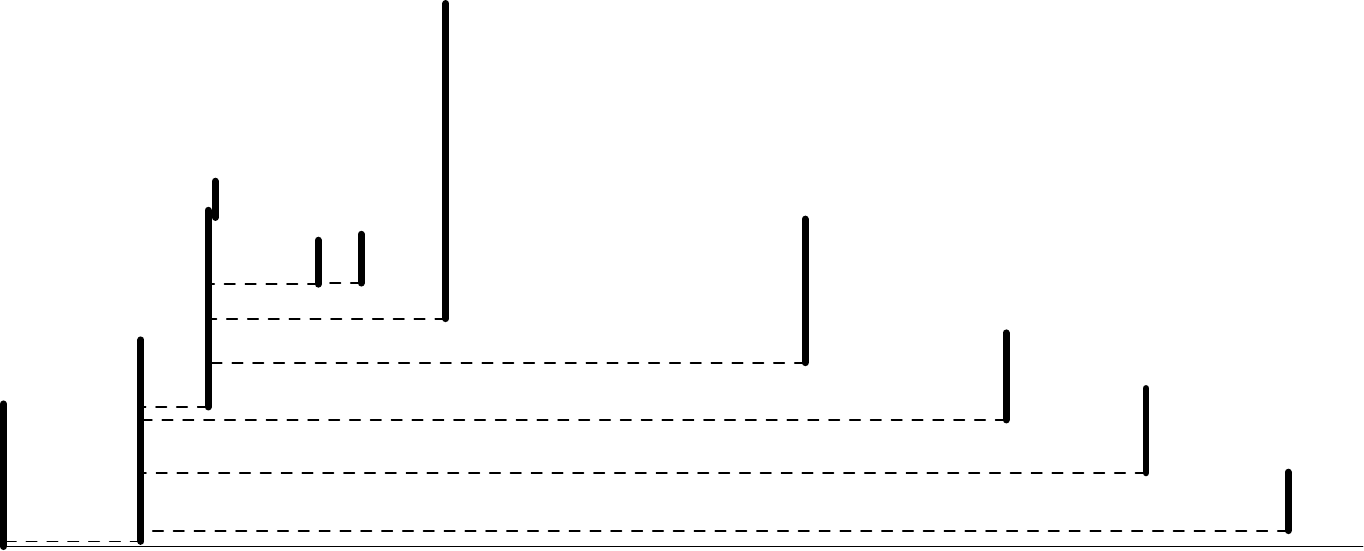}\label{treeModel1}}
\end{center}
\caption{Trajectory of a killed compound Poisson process with drift -1 (A) and of the tree it codes (B). 
The jumps in (A) correspond to the vertical segments in (B) to be joined through the horizontal dashed lines. }
\label{DiscreteChronologicalTree}
\end{figure}
From  a glance at this figure, 
the reader might note that the generation of the individual visited at time $t$ 
equals the number of subtrees grafted to the left of its ancestral line 
(in the figure, there is one such subtree for each dashed horizontal line). 
We can then count the sizes of the successive generations; 
in Figure \ref{DiscreteChronologicalTree}, the successive generation sizes are $1, 4,4$ and $1$. 
(As noted in \cite{MR2599603}, the sizes of succesive generations in the compound Poisson case 
correspond to the well known Galton-Watson process.) 
In analogy, \cite{MR1954248} define the \emph{height process} of a subcritical L\'evy process 
(or the associated TOM tree). 
Let $\Psi$ be the Laplace exponent of an infinite variation spectrally positive L\'evy process 
which is not a subordinator. 
We now assume Grey's condition on the Laplace exponent
\begin{description}
\item[Hypothesis (G)]$\int_{0+}\frac{1}{\imf{\Psi}{\lambda}}\, d\lambda<\infty$.
\end{description}Let $X^r$ be the contour of the truncation 
of a splitting tree $S$ with \emph{law} $\eta^\Psi$ at height $r$ and let $\phi^r$ be its exploration process. 
Based on \cite{MR1954248}, we will obtain 
the existence  of a norming function $a(h)$ such that  
there is a continuous process $H^r$ which agrees with
\begin{linenomath}
\begin{align*}
\label{genealogyCodingProcessInformalDefinition}
t\mapsto
&\liminf_{h\to 0}\frac{1}{\imf{a}{h}} \#\set{\text{Subtrees of $S^r$ to the left of $[\rho,\imf{\phi^r}{t}]$ of height greater than $h$}}. 
\end{align*}\end{linenomath}on a (random) dense set. By properties of L\'evy processes, 
the above limit can be expressed in terms of local times, and hence equal to
\begin{linenomath}\begin{esn}
\liminf_{k\to\infty}\frac{1}{\eps_k}\int_0^t \indi{X^r_s- \underline X^r_{[s,t]}\leq \eps_k}\, ds. 
\end{esn}\end{linenomath}for any sequence $\eps_k$ decreasing to $0$. 
We will call $H^r$ the \defin{genealogy coding process} of $X^r$. 
The quantity $H^r_t$ is our proxy for the generation of the individual visited at time $t$ in the tree coded by $X^r$. 
In the (sub)critical case there is no need to truncate to define a height process $H$, 
which codes a tree $\Gamma$ 
and has been called the \defin{L\'evy tree} in \cite{MR2147221}. 
In the supercritical case the process $H^r$ is then the coding function of a compact real tree. 
The family $\paren{H^r, r\geq 0}$ is compatible under pruning (cf. Lemma \ref{isometricEmbeddingBetweenPartsOfSupercriticalLevyTreeLemma}), so that the sequence of trees 
$\Gamma^r$ that they encode is increasing
(in the sense that $\Gamma^r$ can be embeded in $\Gamma^{r'}$ if $r\leq r'$). 
We will conclude the existence of a limit tree $\Gamma$, 
which we call the genealogical tree associated to our splitting tree. 
The law of $\Gamma$, 
denoted $\gamma$, 
can be decomposed as $\gamma^{\text{c}}+b\gamma^{\text{lc}}$, 
where $\gamma^{\text{c}}$ is the restriction of $\gamma$ to compact trees 
(and is the law of the tree coded by the height process under $\nu^\#$), 
while $\gamma^{\text{lc}}$ is the normalized restriction of $\gamma$ to non-compact trees.

Other articles generalizing L\'evy trees to the supercritical setting are \cite{MR2322700}, \cite{MR2437534} and \cite{MR2962090}. 
In the first one, the authors construct them as limits of Galton-Watson trees consistent under Bernoulli leaf percolation, while in the second and third the construction is carried out by relating the locally compact trees to the compact ones via Girsanov's theorem. 
However, the possibility of studying the height process of a sequence of L\'evy-like processes had not been considered before.

To state a Ray-Knight theorem, 
consider the process\begin{linenomath}\begin{esn}
Z^1_a=
\#\set{\sigma\in \Gamma: \sigma \text{ has an infinite line of descent and }\imf{d}{\sigma,\rho}=a}. 
\end{esn}\end{linenomath}The above quantity is finite by local-compactness. 
Recall that $\pi$ stands for the L\'evy measure of $\Psi$ and $\beta$ for its Gaussian coefficient. 
\begin{theorem}
\label{RayKnightTheorem}
Under $\gamma^{\text{lc}}$, 
the process $Z^1$ is a continuous-time non-decreasing branching process with values in $\na$ 
and jumps in $\set{1,2,,\ldots}$ which starts at $1$. 
Its jump rate from $j$ to $j+k$ (where $k\geq 1$) equals
\begin{linenomath}\begin{esn}
j\bra{\indi{k=1}\beta b+ \int_0^\infty \frac{b^k z^{k+1}}{(k+1)!}e^{-b z}\, \imf{\pi}{dz}}. 
\end{esn}\end{linenomath}

Furthermore, if $\imf{\delta}{\sigma}=\imf{d}{\rho,\sigma}$, then  the random measure $\mu \circ \delta^{-1}$
admits a \cadlag\ density $Z^2$. 
Finally, the process $Z=\paren{Z^1,Z^2}$ 
is a two-type branching process with values in $\na\times [0,\infty)$ started at $(1,0)$. 
Let $\se_{(n,z)}$ be its law when started at $(n,z)$. 
Then $Z$ is characterized by
\begin{lesn}
\left.\frac{d}{dt}\right|_{t=0}\imf{\se_{(n,z)}}{s^{Z^1_t}e^{-\lambda Z^2_t}}
=e^{-\lambda z}s^n\bra{z\imf{\Psi^\#}{\lambda}}+e^{-\lambda z}s^{n-1} n\frac{1}{b}\bra{\imf{\Psi}{\lambda +b(1-s)}-\imf{\Psi}{\lambda+b}}. 
\end{lesn}
\end{theorem}

Two-type branching processes with state-space $[0,\infty)^2$ were introduced in \cite{MR0234531} 
and form part of the affine processes of \cite{MR1994043}. 
In \cite{MR3689968}, 
they have been given a time-change representation which gives insight into their infinitesimal behavior. 
Indeed, 
once we note that $Z^2$ does not influence the behavior of $Z^1$ 
(since non-prolific individuals cannot give rise to prolific ones), 
we see that there exist two independent L\'evy processes 
$X^1=\paren{X^{1,1}, X^{1,2}}$ and $X^2$ 
(with values in $\na\times [0,\infty)$ and $\re$ respectively) 
such that 
$Z$
 has the same law as the unique solution to
 \begin{linenomath}
\begin{align*}
Z^1_t
&=1+X^{1,1}_{\int_0^t Z^1_s\, ds}
&
Z^2_t&=X^{2}_{\int_0^t Z^2_s\, ds}+X^{1,2}_{\int_0^t Z^1_s\, ds}. 
\end{align*}\end{linenomath}%
Note that $Z^1$ has pathwise constant trajectories. 
The link between the infinitesimal behavior of $X$ and $Z$ is as follows: 
if $Z$ is started at $(k,z)$ then, as $t\to 0$, $Z_t$ behaves as 
$X^{k,z}_t=(k+X^{1,1}_{kt},z+X^2_{zt}+X^{1,2}_{kt})$. 
This can be made precise by comparing the derivatives of their semigroups at zero, 
at least for functions whose second derivative is continuous and bounded. 

The quantities
\begin{linenomath}
\begin{align*}
\imf{\Psi^1}{\lambda_1,\lambda_2}&=-\log\esp{e^{-\lambda_1 X^{1,1}_1 -\lambda_2 X^{1,2}_1}}
\intertext{and}
\imf{\Psi^2}{\lambda}&=-\log\esp{e^{-\lambda X^2_1}}
\end{align*}\end{linenomath}(which govern the infinitesimal behavior of $X^1$ and $X^2$) 
are called the branching mechanisms of the two-type branching process $Z$ 
and determine the process uniquely. 
In the setting of Theorem \ref{RayKnightTheorem}, $\Psi^2=\Psi^{\#}$ 
while $X^1$ has drift coefficient $\paren{0,2\beta}$ and its L\'evy measure equal to the sum $\beta b\delta_{(1,0)}+\pi^{f}+\pi^{i}$, where $\pi^f$ is responsible the common finite-activity jumps of $X^{1,1}$ and $X^{1,2}$, while $\pi^i$ is responsible for the infinite activity jumps of $X^{1,2}$. 
We have the explicit expressions
\begin{linenomath}\begin{align*}
\imf{\pi^f}{dk,dx}
=\sum_{l=1}^\infty e^{-bx} b^l\frac{x^{l+1}}{(l+1)!}\, \imf{\delta_l}{dk}\imf{\pi}{dx}
\quad\text{and}\quad
\imf{\pi^i}{dx}= \frac{1-e^{-bx}}{b}\, \imf{\pi}{dx}. 
\end{align*}\end{linenomath}

The above two-dimensional branching process is exactly the one obtained by \cite{MR2455180} 
in their study of the prolific individuals in continuous-state branching (CB) processes 
with branching mechanism $\Psi$. 
The aforementioned work was aimed at extending the well known decomposition 
of a supercritical Galton-Watson process 
in terms of its individuals with infinite and finite lines of descent recalled in Subsection \ref{motivationSubsection}. 
The two-dimensional branching process is also implicit in the work \cite{MR2322700} 
where the authors construct supercritical L\'evy trees by means of increasing limits 
of discrete trees consistent under Bernoulli leaf percolation. 
We have therefore obtained chronological and genealogical 
interpretations of the prolific individuals and an independent construction of supercritical L\'evy trees. 
Superprocess versions of the prolific skeleton decomposition can be found in \cite{MR2794978}, \cite{MR3330813} and the references therein. 

In order to make the link between supercritical CB processes 
and our construction of supecritical L\'evy trees more explicit, 
we will obtain a version of Theorem \ref{RayKnightTheorem} 
in which we obtain a $\imf{\cb}{\Psi}$ process starts at $x$. 
For this, let $x>0$ and, considering the interval $[0,x]$ as a compact TOM tree (to be rooted at $0$). 
Now define a probability measure $\eta_x^{\Psi}$ 
on locally compact TOM trees 
by grafting to the right of $[0,x]$ trees $c_n$ at height $x_n$ 
where $\paren{x_n,c_n}$ are the atoms of a Poisson random measure on $[0,x]$ 
with intensity $\leb\times \eta^{\Psi}$. 
As before, we will first define the height process of the truncated contour $H^r$ under $\eta_x^\Psi$, 
show that these continuous processes code a collection of growing TOM trees, 
hence showing the existence of a limiting TOM tree $\Gamma^x$. 
The statement features a continuous-branching process with branching mechanism $\Psi$, \imf{\cb}{\Psi},  
started at $x$. 
As in the above discussion of the two-dimensional case, 
this process can be represented as the unique solution to
\begin{linenomath}\begin{esn}
Z_t=x+X_{\int_0^t Z_s\, ds}
\end{esn}\end{linenomath}where $X$ is a spectrally positive L\'evy process with Laplace exponent $\Psi$. 
For background on these representations of continuous branching processes, 
the reader is referred to \cite{MR0208685}, \cite{1978Helland} and  \cite{MR2592395} for the monotype case without immigration, \cite{MR3098685} for the monotype case with immigration and \cite{MR3449255} and \cite{MR3689968} for the multitype cases. 

\begin{corollary}[Ray-Knight theorem for supercritical L\'evy trees]
\label{RayKnightCorollary}
Let $\Psi$ be a supercritical Laplace exponent which satisfies Hypothesis \defin{G}. 
Under $\eta_x^\Psi$, 
the measure $\mu\circ \delta^{-1}$ 
admits a \cadlag\ density $Z$. The process $Z$ is a $\imf{\cb}{\Psi}$ which starts at $x$.
%
\end{corollary}

\subsection{Organization}
\label{organizationSubsection}

Section \ref{prolificSection} is devoted to the study of infinite lines of descent in the deterministic setting and to the proof of Proposition \ref{backboneProposition}. 
Then, the results are taken to the random setting of splitting trees in Section \ref{splittingBackboneSection} which features a proof of Theorem \ref{SupercriticalTreeFromGraftingTheorem} and Corollary \ref{YuleTreeCorollary}. 
Section \ref{heightSection} constructs the genealogical tree associated to supercritical splitting trees. 
Finally, Section \ref{rkSection} contains a proof of the Ray-Knight type theorem stated as Theorem \ref{RayKnightTheorem}. 

\section{The prolific skeleton on a locally compact TOM tree}
\label{prolificSection}

Let $\tr{c}=\paren{\paren{\tau,d,\rho}, \leq, \mu}$ be a locally compact TOM tree. 

We will now give a genealogical structure to the infinite lines of descent. 
Let\begin{lesn}
\mc{I}=\set{\sigma\in\tau: \sigma\text{ has an infinite line of descent}}.
\end{lesn}%
\begin{lemma}
\label{nonCompactnessAndInfiniteLinesLemma}
$\mc{I}$ is empty if and only if $\tau$ is compact. 
Otherwise, $\mc{I}$ is a non-compact connected subset of $\tau$ containing the root which inherits the structure of a locally compact TOM tree when equipped with Lebesgue measure. 
\end{lemma}
\begin{proof}
Obviously $\mc{I}$ is empty when $\tau$ is compact. 
When $\tau$ is not compact, then the sphere $S_r=\set{\sigma\in\tau: \imf{d}{\sigma,\rho}=r}$ is non-empty for any $r\geq 0$. 
Let us define\begin{linenomath}\begin{esn}
S^n_r=\set{\sigma\in S_r: \exists \sigma^n\in S_n, \sigma\in[\rho,\sigma_n]}.
\end{esn}\end{linenomath}%
Local compactness implies that $S^n_r$ is finite; it is non-empty since otherwise $S_{r+n} $ would be empty, implying that $\tau$ is compact. 
Note that\begin{linenomath}\begin{esn}
S^{n+1}_r\subset S^n_r\subset S_r.
\end{esn}\end{linenomath}%
Since $S_r$ is compact, by the Hopf-Rinow theorem, then  $S^\infty_r=\bigcap_n S^n_r$ is non-empty and finite. 
Let $\sigma_r$ be the first element of $S^\infty_r$. 
We now prove, by contradiction, that $\sigma_r\preceq\sigma_{r'}$ if $r\leq r'$. 
Indeed, if $\sigma_r\not\preceq\sigma_{r'}$, we can construct, 
by definition of $S_r^\infty$, an element $\tilde\sigma_{r'}\in S_{r'}^{\infty}$ such that $\sigma_r\preceq \tilde \sigma_{r'}$. 
By definition, $\sigma_{r'}<\tilde\sigma_{r'}$. 
However, if we now define $\tilde \sigma_{r}$ as the unique element in $[\rho, \sigma_{r'}]$ at distance $r$ from the root, then (as $\sigma_r\wedge\tilde\sigma_r\neq \sigma_r$), 
$\tilde \sigma_r<\tilde \sigma_{r'}\leq \sigma_{r}$, by \defin{Or2}, which contradicts the definition of $\sigma_r$. 
Hence, $r\mapsto \sigma_r$ is an isometry from $[0,\infty)$ to $\tau$ and by construction $\imf{d}{\rho,\sigma_r}=r$, which increases with $r$, so that $\mc{I}$ is non-empty. 

To see that $\mc{I}$ is connected, it suffices to note that any isometry from $[0,\infty)$ into $\tau$ can be extended to an isometry which contains the root. 
Hence, $\mc{I}$ can be considered a (locally compact) real tree, which can be given a total order by restricting the total order on $\tau$. 
We will give it Lebesgue measure for coding purposes, 
since the measure $\mu$ on our tree $\tau$ might assign zero mass to $\mc{I}$. 
\end{proof}

We will now see that $\mc{I}$ has the structure of a plane tree whose individuals live indefinitely and have associated to them a sequence of birth times. 
\begin{proof}[Proof of Proposition \ref{backboneProposition}]
Note that, for any $r>0$, there are only a finite number of elements of $\mc{I}$ at distance $r$ from $\rho$. 
(Otherwise, there would be an accumulation of long branches, contradicting local compactness). 
This quantity is positive if $\tau$ is non-compact and zero otherwise. 
We denote by $\mc{I}_r=\mc{I}\cap S_r$. 
Let $\sigma_{\emptyset}^r$ be the first element in $\mc{I}$ at distance $r$ from $\rho$; in the proof of Lemma \ref{nonCompactnessAndInfiniteLinesLemma}, we have seen that $r\mapsto \sigma_{\emptyset}^r$ is an infinite line of descent if $\tau$ is non-compact. 
If $\sigma$ belongs to any infinite line of descent and $r=\imf{d}{\sigma,\rho}$, then either $\sigma=\sigma_\emptyset^r$ or $[\sigma, \sigma\wedge \sigma_\emptyset^r)>\sigma^r_\emptyset$. 
So, $\mc{I}_\emptyset=\set{\sigma_\emptyset^r: r\geq 0}$ can be thought of as the first infinite line of descent. 
If $\mc{I}=\mc{I}_\emptyset$, we will call our tree a sin tree (the nomenclature for single infinite end tree as coined in \cite{MR1102319}) and set $\mc{\tau}_I=\set{\emptyset}$.
Otherwise, consider the connected components of $\tau\setminus\mc{I_\emptyset}$ which intersect $\mc{I}$. 

If $\tilde\tau$ is such a connected component and $\tilde\sigma\in \tilde\tau$, 
let $A=\set{t\geq 0: \imf{\phi_{\rho,\tilde\sigma}}{t}\in \mc{I}_\infty}$. The set $A$ is non-empty since $0\in A$. 
If $t=\sup A$, then $t\in A$ since $I_\emptyset$ is closed. 
Let $\tilde \rho=\imf{\phi_{\rho,\tilde\sigma}}{t}$. 
We now assert that $\tilde \rho$ is independent of the element $\tilde\sigma\in\tilde\tau$ that we considered. 
Indeed, if $\tilde\sigma_1<\tilde\sigma_2\in\tilde\tau$ gave rise to $\rho_1\neq \rho_2$, then $\rho_1<\rho_2$ by \defin{Or2} and this would create a cycle since we would be able to go from $\tilde\sigma_1$ to $\tilde\sigma_2$ inside of $\tilde \tau$ (by connectedness of components) or going from $\tilde\sigma_1$ to $\tilde\rho_1$, going up from $\tilde\rho_1$ to $\tilde\rho_2$ inside $\mc{I}_\emptyset$, and then from $\tilde\rho_2$ to $\tilde \sigma_2$.  
%
Any path from $\tau\setminus\tilde\tau$ into $\tilde\tau$ must therefore pass through $\tilde \rho$ 
(otherwise there would be cycles). 
Then $\tilde\tau\cup\set{\tilde \rho}$ is a TOM tree; to prove it we just need to see that $\tilde\tau\cup\set{\tilde \rho}$ is closed. 
Let $\sigma_n$ be a sequence of $\tilde\tau$ converging to $\sigma\in\tau$. 
If $\sigma$ did not belong to  $\tilde\tau\cup\set{\tilde\rho}$, then the path $[\sigma_n,\sigma]$  
would hence have to contain $\tilde \rho$. This would  imply the inequality $\imf{d}{\sigma_n,\sigma}\geq \imf{d}{\tilde\rho,\sigma}>0$, which is incompatible with $\sigma_n\to\sigma$. Hence, components $\tilde \tau$ of $\tau\setminus\mc{I}_\emptyset$, when rooted at their corresponding $\tilde\rho$ and restricting order and measure to them, become TOM trees. We will call these the rooted components. 

We now proceed to order the rooted components. 
If $\tau_1$ and $\tau_2$ are two rooted components of $\tau\setminus\mc{I}_\emptyset$ (say rooted at $\rho_1$ and $\rho_2$), 
consider $\sigma_i,\tilde\sigma_i\in\tau_i$. 
Then $\sigma_1\leq \sigma_2$ implies $\tilde \sigma_1\leq \tilde\sigma_2$. 
Indeed, note that $\sigma_1\wedge\tilde\sigma_1<\sigma_2\leq \sigma_2\wedge \tilde \sigma_2$ since $\sigma_1\wedge \tilde\sigma_1\in\tau_1$ and so $\sigma_1\wedge \tilde\sigma_1\neq \sigma_1\wedge \sigma_2$ so that \defin{Or2} implies $\sigma_1\wedge \sigma_2< \sigma_2\leq \sigma_2\wedge \tilde\sigma_2$. 
But then the inequality $\tilde\sigma_1> \tilde\sigma_2$ would imply the contradictory inequality $\sigma_1\wedge \sigma_2< \sigma_2\leq \sigma_2\wedge \tilde\sigma_2$. 
We conclude that $\tilde\sigma_1<\tilde\sigma_2$. 

Hence, we can order the rooted components of $\tau\setminus\mc{I}_\emptyset$ which intersect $\mc{I}$, say as $\tau_1,\tau_2,\ldots$. 
We label them by increasing height of their root and in case of components $\tau_i,\tau_j$ with the same root, we impose that $i\leq j$ implies $\tau_i<\tau_j$, since there is only a finite number of components of $\tau\setminus\mc{I}_\emptyset$ intersecting $\mc{I}$ and sharing the same root by local compactness. 
(The ordering between two components is clear if their roots are different, which is the case when $\tau$ is binary). 
Let $k_\emptyset$ be the quantity of such connected components ($k_\emptyset$ can be zero or infinite). 
Then the first generation of $\tau_1$ consists of  $1,\ldots, k_\emptyset$ if $k_\emptyset$ is finite and of $\z_+$ otherwise. 
If $\tau_i$ is rooted at $\rho_i$, we set $\alpha_i=\imf{d}{\rho,\rho_i}$. 
Note that if $k_\emptyset=\infty$ then $\alpha_i$ is increasing and converges to $\infty$.  
Indeed, by local compactness, only a finite number of the $\alpha_i$ can belong to a compact interval of $\re_+$. 

Hence, starting from any non-compact TOM tree $\tau$, we have built its first infinite line of descent $\mc{I}_\emptyset$ and provided a particular labeling for the components of $\tau\setminus \mc{I}_\emptyset$ which intersect $\mc{I}$. 

We now proceed recursively. 
Starting from $\tau_\emptyset=\tau$, 
we consider its first infinite line of descent, with image $\mc{I}_\emptyset$ as well as the labeled components $\tau_1,\tau_2,\ldots$. 
Then, on each one of the components, we repeat the procedure. 
The image of the first infinite line of $\tau_u$ is denoted $\mc{I}_u$. 
The root of $\tau_u$ is called $\rho_u$ and we let $\alpha_u=\imf{d}{\rho,\rho_u}$. 
We let $k_u$ be the quantity of connected components of $\tau_u\setminus\mc{I}_u$ which intersect $\mc{I}$. 
If $k_u=0$, we have finished exploring this part of the tree. 
If $k_u>0$, the rooted connected components of $\tau_u\setminus\mc{I}_u$, labeled in our particular way, will be denoted $\tau_{ui},1\leq i\leq k_u$, and we now explore these. 
Notice that, by construction, $\mc{I}_u\cap \mc{I}_{ui}=\set{\rho_{ui}}$. 
The tree $\tau_I$ consists of the labels used for the lines of descent. 

Let us now show that $\mc{I}=\bigcup_{u\in\tau_I}\mc{I}_u$. 
This follows from the more general equality\begin{linenomath}\begin{equation}
\label{infiniteLinesOfDescentOnSphereEquality}
S_r\cap \mc{I}=\bigcup_{u\in\tau_I, \alpha_u\leq r}S_r\cap \mc{I}_u,
\end{equation}\end{linenomath}which will be proven by induction on the (finite) quantity of elements of $S_r\cap \mc{I}$. 
When $S_r\cap \mc{I}$ has only one element, this is, by construction, the individual of $\mc{I}_\emptyset$ at height $r$, so that $S_r\cap\mc{I}= S_r\cap\mc{I}_\emptyset$. 
Suppose that the equality \eqref{infiniteLinesOfDescentOnSphereEquality} holds for any TOM tree and any $r\geq 0$ whenever $S_r\cap \mc{I}$ has less than $n$ elements. 
If for our tree $\tau$, $S_r\cap \mc{I}$ has $n+1$ elements, then one (and only one) of these elements belongs to $\mc{I}_\emptyset$. 
The others belong to rooted connected components of $\tau \setminus \mc{I}_\emptyset$, say with labels $1,\ldots, k$ such that $\alpha_{i}\leq r$. 
Denote these components by $\tau_{1},\ldots, \tau_k$. 
By construction, the infinite lines of descent of $\tau_i$ are $\paren{\mc{I}_{iu}: iu\in\tau_I}$. 
If $\mc{I}^i$ denotes individuals with infinite lines of descent of $\tau_i$ and $S^i_r$ denotes individuals in $\tau_i$ at distance $r-\alpha_{u_i}$ from $\rho_{u_i}$, note that $\mc{I}^i\cap S^i_r$ has at most $n$ elements, so that from our induction hypothesis we get
\begin{linenomath}
\begin{esn}
\mc{I}^i\cap S^i_r
=\bigcup_{\substack{iu\in\tau_I\\ \alpha_{iu}\leq r}} \mc{I}_{iu}\cap S_r.
\end{esn}\end{linenomath}Hence,
\begin{linenomath}\begin{esn}
\mc{I}\cap\ S_r
= \mc{I}_\emptyset\cap S_r\cup \bigcup_{\substack{i\leq k\\ iu\in\tau_v\\ \alpha_{iu }\leq r}} \mc{I}_{iu }\cap S_r
=\bigcup_{\substack{u\in\tau_I\\ \alpha_u\leq r}}S_r\cap \mc{I}_u. \qedhere
\end{esn}\end{linenomath}
\end{proof}

\section{Backbone decomposition of supercritical splitting trees}
\label{splittingBackboneSection}
In this section, we analyze the laws $\Upsilon$ and $\Upsilon_{\text{tree}}$ with the aim of proving Theorem \ref{SupercriticalTreeFromGraftingTheorem}. 
We first prove that under $\Upsilon$, there exists a unique infinite line of descent that contains the root. 
Then, we consider the measure $\Upsilon_{\text{tree}}$ and prove that the infinite lines of descent are a Yule tree and move on to the proof of Theorem \ref{SupercriticalTreeFromGraftingTheorem}. 
\subsection{Infinite lines of descent under $\Upsilon$}
Recall that the probability measure $\Upsilon$  is the limit of trees with laws $\Upsilon^r$ 
coded by the concatenation of the post-minimum process of a $\Psi$-L\'evy process 
(time-changed to remain below $r$) 
followed by an independent $\Psi^\#$-L\'evy process started at $r$ and time-changed to remain below $r$ until one of them reaches zero. 
However, in order to access the infinite line of descent, 
we need to define the trees with laws $\Upsilon^r$ on a unique probability space so that the tree with law $\Upsilon$ becomes its pointwise direct limit. 
\begin{proposition}
\label{uniquenessInfiniteLineOfDescentProposition}
Let $S$ be a tree with law $\Upsilon$. 
Then $S$ admits a unique infinite line of descent. 
\end{proposition}

\begin{proof}
Let $X^0,X^1,\ldots$ be independent processes, 
where $X^0$ has the law $\p_0^{\rightarrow}$ and for $i\geq 1$,  
the law of $X^i$  is the image under $\p_i^\#$ by killing upon reaching $i-1$. 
We then let $X^{i,n}$ equal $X^i$ time-changed to remain below $n$, 
with the understanding that if $i\geq n$ then this is the trivial trajectory 
which is ignored when referring to it for concatenation purposes. 
Finally, we just let $Y^n$ equal the concatenation of $X^{0,n}$, $X^{n,n}, \ldots, X^{1,n}$. 
Assume that the processes are concatenated at times $T_1^n<T_2^n<\cdots< T^n_n$ and that $Y^n$ is defined until $T^n_{n+1}$. 
Note that the sequence of processes $Y^n$ is (pointwise) consistent under time-change. 
We then let $S^n$ be the tree coded by $Y^n$ and let $S$ be the pointwise direct limit of the sequence $(S^n)$, consisting of equivalence classes consisting of elements the type $(n,\sigma)$ with $\sigma\in S^n$ (as explained in Subsection \ref{codingSequenceSubSubSection}). 
In what follows we identify  $\sigma\in S_n$  and the class of $(n,\sigma_n)$. 
Note that the law of $S$ is $\Upsilon$. 
We first show that $S$ has at least one infinite line of descent. 
Indeed, consider first the path to the root from $\sigma^n_n=[T^n_1]_{Y^n}$ (considered as an element of $S$): 
this consists of individuals
\begin{linenomath}
\begin{esn}
[\rho, \sigma^n_n]
=\set{[t]_{Y^n}: t\geq T^n_1 \text{ and } Y^n_{t}=\underline Y^n_{[T^n_1,t]}}
\end{esn}\end{linenomath}as explained when introducing the tree coded by a function in \cite{splitting1}. After $T^n_1$, $Y^n$ reaches level $i<n$ at $T^n_i$. 
It follows that $\sigma^n_i=[T^n_i]_{Y^n}\in [\rho,\sigma^n_n]$. 
Note that $\sigma^n_{i}\preceq\sigma^n_{i+1}$. 
Hence, $\mc{I}=\bigcup_n [\rho,\sigma^n_n]$ is an infinite line of descent since $\imf{d}{\rho,\sigma_n^n}=Y^n_{T^n_1}=n\to\infty$ as $n\to\infty$. 

We now prove that  $\mc{I}$ is the unique infinite line of descent. 
Indeed, if $\sigma\in  S$, we consider $\tilde n$ such that $\imf{d}{\sigma,\rho}<\tilde n$ and hence that $\sigma\in S^{\tilde n}$. 
Also, let $ n\geq \tilde n$ be such that the maxima of $X^1,\ldots, X^{\tilde n}$ and of $X^0 $ (until the last time $\Lambda_{\tilde n}$ it visits $[0,\tilde n]$) are less than $n$. 
Suppose that $\sigma=[t]_{Y^n}$. 
We divide into cases depending on if $t<T^n_1$ or not. 
In the first case, note that $Y^{n+m}=Y^{n}=X^0$ on $[0, \Lambda_{\tilde n}]$ and $Y^{n+m}\geq \tilde n$ on $[\Lambda_{\tilde n}, T^{n+m}_1]$. 
If $X^0_t=\underline X^0_{[t,\infty)}$ then $Y^n_{t}=\underline Y^n_{[t, T^n_1]}$ and if we let $\tilde t=\inf\set{s\geq T^n_1: Y^n_s\leq Y^n_t}$ then $[t]_{Y^n}=[\tilde t]_{Y^n}\preceq [T^n_1]_{Y^n}\in   \mc{I}$. 
Otherwise, if $\imf{d}{\sigma,\rho}=X^0_t>\underline X^0_{[t,\infty)}$, then, by the choice of $n$, the subtree above $\sigma$ is compact (and coded by $X^0$ (or $Y^n$) from the first time $X^0$ exceeds $X^0_t$ until the last time it is above that quantity. ) 
Hence, $\sigma$ is not on an infinite line of descent (but attaches to its left). 
When $t\geq T^n_1$, we can analogously divide into the cases depending on if $Y^n_t=\underline Y^n_{[T^n_1,t]}$ or not. 
The proof follows the same line as the one just presented and we one sees that the excursions above the cumulative minimum of the $X^i$ code trees that attach to the right of the infinite line of descent. 
\end{proof}

\subsection{Yule trees and the prolific skeleton decomposition}
The objective of this section is to prove  Theorem \ref{SupercriticalTreeFromGraftingTheorem} and Corollary \ref{YuleTreeCorollary}.  
To this end, we will fix a level $r>0$ and consider the contour of the truncation at level $r$ of the restriction of $\eta^\Psi$ to locally compact trees, whose law was equal to $b\q^{\rightarrow,r}$,  as well as the corresponding contour of the truncation of the $\Upsilon$-tree. 
Theorem \ref{SupercriticalTreeFromGraftingTheorem} will follow once we prove that the above two contours have the same law.

Recall the measure on sin trees $\Upsilon$ defined before the statement of Theorem \ref{SupercriticalTreeFromGraftingTheorem}. 
Let us 
describe the law of the contour process of the image $\Upsilon^r$ of $\Upsilon$ upon truncating at level $r$. 
Recall that the contour process under $\Upsilon^r$ is the concatenation of  $X^{\rightarrow}$ time-changed to remove the part of the trajectory above $r$ 
and a L\'evy process with Laplace exponent $\Psi^{\#}$ started at $r$, reflected below level $r$  and killed upon reaching zero. 
Because of our description of the infinite line of descent under $\Upsilon$, 
we might think of this truncated sin tree as the (vertical) interval $[0,r]$ 
where we graft trees to the left and to the right; 
the left corresponding to the process $X^{\rightarrow}$ and the right to the subcritical L\'evy process (both time-changed to remain below $r$). 
Hence, to the right of the interval $[0,r]$, we just graft trees $f$ at $s$ where $\paren{s,f}$ is a Poisson random measure with intensity $\imf{\leb}{ds}\otimes \imf{\nu^{\#,r-s}}{df}$, where $\nu^{\#,r}$ is the excursion measure corresponding to the exponent $\Psi^{\#}$ and then time-changed to keep the process below $r$. 
Let us consider this description when passing to a $\Upsilon_\text{tree}$ truncated at level $r$, say $I^r$. 
By construction, $I^r$ can be thought of as the interval $[0,r]$, whose tip is denoted $\sigma_0$, where the left is coded by $X^{\rightarrow}$ (time-changed) and to the right we graft trees as under $\Upsilon$ and additionally graft truncated locally compact trees $(S_k,I_k)$ at the atoms of a Poisson random measure with intensity $b\, ds\otimes \imf{\Upsilon^{r-s}_{\text{tree}}}{df}$. We will suppose that $S_1>S_2>\cdots$, so that $I_1$ is the tree that is grafted to the right of $[0,r]$ farthest from the root and conditionally on $S_1=s$, $I_1$ has law $\Upsilon^{r-s}_{\text{tree}}$ and tip $\sigma_1$. 
If we graft no truncated locally compact trees, then the right of $I^r$ is coded by a L\'evy process with exponent $\Psi^\#$, started at $r$, killed upon reaching zero and time-changed to remain below $r$. 
Otherwise, 
the right of $\sigma_0$ is coded by 3 different parts (in their correct chronological order): 
first, what happens between $\sigma_0$ and $\sigma_0\wedge \sigma_1$, 
then between $\sigma_0\wedge \sigma_1$ and $\sigma_1$, 
and finally the right of $\sigma_1$.
Conditionally on $S_1=s$, between $\sigma_0$ and $\sigma_0\wedge \sigma_1$, 
we have a L\'evy process with exponent $\Psi^{\#}$ time-changed to remain below $r$ and killed upon reaching $s$. 
Then, what lies between $\sigma_0\wedge \sigma_1$ and $\sigma_1$ is coded by a process with law $s+X^{\rightarrow}$ time-changed to remain below $r$. 
Since $r-S_1$ is exponential of parameter $b$ when $I_1$ needs to be grafted, then these two pieces, plus the value of $S_1$ can be combined to obtain a L\'evy process with exponent $\Psi$ conditioned to remain above zero (its minimum will be $r-S_1$) and time-changed to remain below $r$. 
Finally, the right of $\sigma_1$ in $I^r$ is divided into the right of $\sigma_1$ in $I_1$ and the right of $\sigma_0\wedge \sigma_1$ in $I^r$. 
However, this has the same law as $I$, 
meaning that we restart with the same procedure. 
Iterating, we see that the coding function for $I^r$ also admits the following description: we start with a process with law $\p^{\rightarrow}$ time-changed to remain below $r$ until its death-time, followed by processes with laws $\p_r$ time-changed to remain below $r$ which will get concatenated until one of them reaches zero. 
We deduce that the coding function for $I^r$ has the same law as the corresponding coding function under $\eta^r$, which concludes the proof of Theorem \ref{SupercriticalTreeFromGraftingTheorem}. 

Regarding Corollary \ref{YuleTreeCorollary}, 
we just note that under $\eta^\Psi$ the infinite lines of descent are non-empty only when the tree is locally compact. 
However, the restriction of $\eta^\Psi$ to locally compact trees is $b\Upsilon_{\text{tree}}$. 
The construction of the latter, 
plus the fact that under $\Upsilon$ there is a unique infinite line of descent thanks to Proposition \ref{uniquenessInfiniteLineOfDescentProposition}, show that the tree of infinite lines of descent under $\Upsilon_{\text{tree}}$ is a Yule tree of birth rate $b$.

\section{The height processes and the genealogical tree associated to supercritical splitting trees}
\label{heightSection}

In this section, we aim at constructing the genealogical tree associated to a supercritical splitting tree. 
This will be accomplished by considering the height processes, 
introduced in \cite{MR1617047} and \cite{MR1954248}, of truncations of splitting trees. 
This provides us with a family of continuous functions coding a growing sequence of compact trees. 
A direct limit construction shows us the existence of a locally compact TOM tree;
the limit tree will be termed the supercritical L\'evy tree since it reduces to the L\'evy tree in the subcritical case. 
Let us now turn to the construction of the height process. 

Recall that if $X$ is any stochastic process and $\paren{\eps_k}$ is any sequence decreasing to zero, 
one can define a measurable version of the height process of $X$, 
denoted $\imf{H^\circ}{X}$, by means of
\begin{linenomath}
\begin{equation*}
\label{DuquesneLeGallNormalization}
\imf{H^\circ }{X}_t=
\liminf_{k\to\infty}\frac{1}{\eps_k}\int_0^t \indi{X_s- \underline X_{[s,t]}\leq \eps_k}\, \d s.
\end{equation*}\end{linenomath}
Then, one defines the height process as a good version of $H^\circ$. 
If $Y^r$ is the contour of a $\Upsilon_\text{tree}$ truncated at height $r$, and assuming that Grey's condition 
\begin{description}
\item[(G) ] $\int^\infty 1/\imf{\Psi}{q}\, \d q<\infty$
\end{description}holds, we now  construct a continuous extension  of  (the restriction of) $\imf{H^\circ}{Y^r}$ (to a random dense set). 

Recall that we assume  that $\Psi$ is supercritical 
and we let $b>0$ denote the positive root of $\Psi$. 

We will also need the Laplace exponent $\Psi^{\#}$ where $\imf{\Psi^{\#}}{q}=\imf{\Psi}{b+q}$. 
Notice that the L\'evy processes corresponding to $\Psi$ and $\Psi^{\#}$ have paths of unbounded variation (because of \defin{G}) and that hence $0$ is regular for both half-lines thanks to Corollary VII.5 of \cite{MR1406564}. 
As referenced in the introduction, in this case,  the infimum of $X$ on any interval $[s,t]$ is achieved continuously at a unique place. 

Fix $r>0$. 
Let $X^1, X^2,\ldots$ be independent processes. 
$X^1$ has law $\p^{\rightarrow}$, while $X^2,X^3,\ldots$ are $\Psi$-L\'evy processes started at $r$ and killed when they reach zero. 

By concatenation, we define the process $Y^r$ as follows. 
First, we define the time-change $C^i$ as the right-continuous inverse of\begin{linenomath}\begin{esn}
A^i_t=\int_0^t \indi{X^i_s\leq r}\, \d{s}. 
\end{esn}\end{linenomath}%
Since each $X^i$ either has finite lifetime or drifts to infinity, we see that $C^i_\infty<\infty$. 
We define $T_i=C^1_\infty+\cdots+C^i_\infty$ and $T_0=0$. 
Next, let $N$ be the first index $i$ such that $X^i\circ C^i$ approaches zero at death time. 
We then define\begin{linenomath}\begin{esn}
Y^r_t=\sum_{i=1}^N \indi{t\in [T_{i-1},T_i)} X^i\circ C^i_{t-T_{i-1}}. 
\end{esn}\end{linenomath}

The process $Y^r$ codes a real tree which has been interpreted as the contour of the chronological tree of a population of individuals which have iid lifetimes and reproduce at constant rate to iid copies of themselves, seen until time $r$. 
The processes $Y^r$ are consistent under time change, so that if $r'\leq r$ then removing the trajectory on top of $r$ from $Y^{r'}$ (with a time-change analogous to the $C^i$) leaves a process with the same law as $Y^r$ (cf. Corollary 8 and Propositon 9 in \cite{splitting1}). 
Hence, we can actually build the processes $Y^r$ on the same probability space so that the time-change consistency is valid pathwise. 
Hence, the trees they code naturally form an increasing family and we can construct from them, by a direct limit construction, a unique locally compact TOM tree  whose truncation at level $r$ is coded by a process with the same law as $Y^r$ and which is not compact.

For (spectrally positive) L\'evy processes satisfying  Grey's condition and in the subcritical case (so under $\p^\#$, say), \cite{MR1954248} construct the so-called \defin{Height process} of $X$, denoted $H$, as a continuous modification of the process $\imf{H^\circ }{X}$, with additional links to the (suitably normalized Markovian) local time $L^{(t)}$ of the time-reversed processes $\hat X^t$ given by $\hat X^t_s=X_{(t-s)-}-X_t$. 
Indeed, according to Lemma 1.4.5 of \cite{MR1954248}, there exists a sequence $\eps_k\downarrow 0$ such that, almost surely, if $s$ is an \defin{upward time} for $X$, meaning that there exists a rational $t>s$ satisfying  $X_{s-}\leq \underline X_{[s,t]}$), we have:
\begin{linenomath}
\begin{equation}
\label{localTimeLemmaFromDLG}
H_s=L^{(t)}_t-L^{(t)}_{t-s}=\imf{H^\circ}{X}_s. 
\end{equation}\end{linenomath}
Note that the (random) set of upward times is dense on $(0,\zeta)$; for example, any jump time is an upward  time and jumps of $X$  are dense under $\p$, $\p^\#$ and $\p^\rightarrow$. 
They will be of fundamental importance in our analysis, since the equality $H_s=H^\circ_s$ is valid for all upward times $s$ under $\p_x$. 
We will have to consider an alternative to modifications for height proceses, 
since we were unable to make them work with time-changes. 
Instead, we will let $H^u$ (or $\imf{H^u}{X}$) denote the restriction of $H^\circ$ to the set of upward times. 
We will construct a continuous extension of $H^u(Y^r)$ and define it as the height process of $Y^r$

To construct a continuous extension of $\imf{H^u}{Y^r}$, we first construct a continuous extension of $\imf{H^u}{X\circ C^r}$ under $\p_{x}^{\#}$, then under $\p_x$ and $\p^{\rightarrow}$, then finally for $Y^r$. 
We simplifly notation in the next proposition by not writing $r$ as a superscript. 

\begin{proposition}
\label{timechangedHeightProcessForReflectedLevyProcessProposition}
Under $\p^{\#}$, 
$\imf{H}{X}\circ C$ is the unique continuous extension of $\imf{H^u}{X\circ C}$.
Additionally, almost surely, if $t$ is upward for $X$ then $\imf{X\circ C}{t-}<r$ and $C_t$ is upward for $X$, so that we have the equality $\imf{H}{X}\circ C_t=\imf{H^\circ }{X\circ C}_t$. 
\end{proposition}
\begin{proof}
	

We first prove that $\imf{H}{X}\circ C$ is continuous. 
%
Since $H$ is continuous, we only to see what happens at the discontinuities of $C$. 
A discontinuity of $C$ at $t$ corresponds to an excursion interval of $X$ above $r$: $X>r$  on $(C_{t-},C_t)$. 
Our aim is to prove that $\imf{H}{X}\circ {C_{t-}}=\imf{H}{X}\circ {C_{t}}$. 
Notice that all excursion intervals can be captured by defining, 
for each rational $u\geq 0$, 
\begin{linenomath}\begin{esn}
d_u=\inf\set{s\geq u: X_t\leq r}\quad\text{and}\quad  g_u=\sup\set{s\leq u: X_s\leq u}. 
\end{esn}\end{linenomath}
Then excursion intervals are of the form $(g_u,d_u)$, whenever $g_u<d_u$ (which happens whenever $X_u>r$). 
Hence, it suffices to prove that $H_{g_u}=H_{d_u}$ for every rational $u$. 
Note that $d_u$ is a stopping time. 
By regularity, for any rational $v>d_u$, we have that $\underline X_{[d_u,v]}<r $. 
Hence we can define $\rho_v$ to be the (unique) instant at which $\underline X_{[d_u,v]}=X_{\rho_v}$ and note that $\rho_v\to d_u$ as $v\to d_u$. 
Also, we can define $\gamma_v=\sup\set{s\leq u: X_{s}\leq X_{\rho_v}}$, so that $X_{\gamma_v-}, X_{d_v-}\leq \underline X_{[\gamma_v,v]}<X_v$ and $\gamma_v\to g_u$ as $v\downarrow d_u$. 
Using \eqref{localTimeLemmaFromDLG},
we see that $H_{\rho_v}=L^v_v-L^v_{v-\rho_v}$ and $H_{\gamma_v}=L^v_v-L^v_{v-\gamma_v}$. 
However, using the support properties of local times (cf. Theorem 4.iii of \cite{MR1406564}), we see that  $L^v$ does not increase on the inverval $[v-\gamma_v,v-\rho_v]$ so that $H_{\rho_v}=H_{\gamma_v}$. 
By continuity of $H$, we see that $H_{g_u}=H_{d_u}$. 

	Suppose now that $t$ is upward for $X\circ C$. 
	Then $\imf{X\circ C}{t-}<r$. 
	Indeed,  this is clear if $\Delta X\circ C_t>0$. 
	On the other hand, if $\Delta X\circ C_t=0$ and $t$ is upward for $X\circ C$, the equality $X\circ C_t=r$ would then imply the existence of $s>t$ such that $X\circ C$ is constant on $[t,s]$, which is impossible thanks to the proof of Proposition 7 of  \cite{splitting1}. 
	By considering these two cases, since $X\circ C(t-)<r$, 
	we deduce the existence of a rational $u>t$ such that $X(C(t)-)\leq \underline{X}_{[C(t),u]}$, 
	so that $C_t$ is upward for $X$. 
	Then, for any $\eps\in (0,r-X\circ C_{t-})$ and $s<t$, the inequality $X_s-\underline X_{[s,C_t]}<\eps$ implies $X_s<r$. 
	Also, we have that $\underline{X\circ C}_{[s,t]}=\underline X_{[C_s,C_t]}$.
	Then, by change of variables and \eqref{localTimeLemmaFromDLG}:
	\begin{linenomath}
	\begin{align*}
	\imf{H^u}{X\circ C}_t 
	&=\imf{H^\circ}{X\circ C}_t\\
	&=\liminf_{k\to\infty}\frac{1}{\eps_k}\int_0^{t} \indi{X\circ C_s-\underline{X\circ C}_{[s,t]}<\eps_k}\, \d{s}
	\\&=\liminf_{k\to\infty}\frac{1}{\eps_k}\int_0^{t} \indi{X\circ C_s-\underline{X}_{[C_s,C_t]}<\eps_k}\, \d{s}
	\\&=\liminf_{k\to\infty}\frac{1}{\eps_k}\int_0^{C_t} \indi{X_s-\underline X_{[s,C_t]}<\eps_k}\indi{X_{s}\leq r}\, \d{s}
	\\&=\liminf_{k\to\infty}\frac{1}{\eps_k}\int_0^{C_t} \indi{X_s-\underline X_{[s,C_t]}<\eps_k}\, \d{s}
	\\&=\imf{H }{X}\circ C_t
	\end{align*}\end{linenomath}
	Finally, $\imf{H^u}{X\circ C}$ is densely defined (since every jump time $t$ of $X\circ C$ is upward and these jump times are dense on the interval of definition of $X\circ C$). Hence, its continuous extension is unique. 
\end{proof}
To explain why we can construct a continuous version of the height process of $X\circ C$ under $\p_x$ and $\p^\rightarrow$, recall that the laws $\p_x$ and $\p_x^{\#}$ are equivalent on $\F_t$ for each $t>0$, 
so that $\imf{H}{X}\circ C$ is a continuous extension of  $\imf{H^u}{X\circ C}$ under $\p_x$. By killing, we see that  $\imf{H^u}{X\circ C}$ admits the continuous extension $H\circ C$ under $\q_x$ (which stands for the image of $\p_x$ under killing when reaching zero) for any $x\geq 0$.

Recall that $\p^{\rightarrow}_x$ is the law of the post minimum $X^\rightarrow$ process under $\p_x$. 
Hence, if $H$ is a continuous extension of  $\imf{H^u}{X}$ and $m$ is the unique time at which $X$ reaches its minimum, then $\tilde H=H_{m+\cdot}$ will be a continuous extension  of $\imf{H^u}{X^{\rightarrow}}$. 
The time-change $C$ is the identity until $X^\rightarrow$ reaches the threshold $r$ after which the process has the same law as $X$ under $\p_r$ conditioned on remaining poisitive. 
So,  $\tilde H\circ C$ is still a continuous extension of $\imf{H^u}{X^\rightarrow\circ C}$. 

We have seen that, for each one of the processes $X^i\circ C^i$, there exists a continuous extension $H^i$ of $\imf{H^u}{X^i\circ C^i}$.
We now construct a continuous extension of  $\imf{H^u}{Y^r}$. 
\begin{proposition}
Define $H$ as follows: 
for any $i\geq 1$ and $t\in [T_{i},T_{i+1})$, 
let
\begin{linenomath}\begin{esn}
g_t=\sup\set{s\leq T_{i}: Y^r_s\leq \underline Y^r_{[T_i,t]}}
\end{esn}\end{linenomath}and define
\begin{linenomath}
\begin{esn}
H=H^1\text{ on }[0,T_1]
\quad \text{and, recursively, }\quad
H_t=H^{i+1}_{t-T_i}+H_{g_t}\text{ for }i\geq 1\text{ and }t\in [T_i,T_{i+1}). 
\end{esn}\end{linenomath}Then, $H$ is a continuous extension of  $\imf{H^u}{Y^r}$. 
\end{proposition}
\begin{proof}
Recall that $H^i_0=0=\lim_{t\to 0+}H^i_{t}$. 
Also, $g_{T_i}=T_i$. 
We then see that $H$ is continuous at each $T_i$. 
To prove that $H$ is continuous at $t\in (T_i,T_{i+1})$ for some $i\geq 1$, it suffices to show that $t\mapsto H_{g_t}$  is continuous there. 
However, let us note that $t\mapsto g_t$ is decreasing and c\`agl\`ad on $(T_i,T_{i+1})$. 
It might then happen that $g_{t+}<g_t$ and they fall on different intervals $(T_k,T_{k+1})$ and $(T_l,T_{l+1})$ with $k<l<i$. 
However, by definition, this implies\begin{linenomath}\begin{esn}
H_{g_t}=H_{g_{t+}}+H^{l+1}_{g_t-T_l}
\end{esn}\end{linenomath}and since the minimum of $Y^r$ on $(T_l,T_{l+1})$ 
is attained at $g_t$, then $H^{l+1}_{g_t-T_l}=0$ and $H_{g_t}=H_{g_{t+}}$. 
Indeed, note that for any rational $v\in (g_t,T_{l+1})$, 
formula \eqref{localTimeLemmaFromDLG} gives $L^{v}_v-L^v_{v-g_t}$ 
and that by support properties of local times, $L^v$ is constant on $[v-g_t,v]$. 
It remains to consider the case when $g_t,g_{t+}\in (T_l,T_{l+1})$ for some $l<i$. 
In this case, we note that $(C^{l+1}_{g_{t+}-T_l}, C^{l+1}_{g_t-T_l})$ is an excursion interval of $X^l$ above its future minimum process so that in particular $g_t-T_l$ is upward. 
Hence, the height process of $X^l$ is constant on that interval 
(again by \eqref{localTimeLemmaFromDLG} and support properties of local times), 
which implies $H_{g_{t}}=H_{g_{t+}}$. 
We conclude that $H$ is continuous. 

Let us now prove that $H$ is an extension of $\imf{H^u}{Y^r}$. 
We need to prove that 
$H_t
=H^{\circ }(Y^r)_t$ for every upward time $t\in [T_i,T_{i+1})$  for $Y^r$ and for any $i$. 

On $t\leq T_1$, we see that $H_t=H^1_t=H^{\circ }(Y^r)_t$ 
if $t$ is upward for $Y^r$ by definition of $H$ and Proposition \ref{timechangedHeightProcessForReflectedLevyProcessProposition}. 
To proceed by induction, assume that for some $j\geq 1$, 
$H_t=\imf{H^\circ}{Y^r}_t$ 
if $t\leq T_j$ and $t$ is upward for $Y^r$. 
If we now work on the set $t\in (T_{j},T_{j+1})$, 
note that $g_t\in [T_i,T_{i+1})$ for some $i<j$. 
Note that both $H$ and $\imf{H^\circ}{Y^r}$ can be decomposed  as\begin{linenomath}\begin{align}
H_t&=H_{g_t}+H^{j+1}_{t-T_j}
\intertext{and}
\imf{H^\circ}{Y^r}_t
=\imf{H^\circ}{Y^r}_{g_t}
&+\liminf_{k\to\infty}\frac{1}{\eps_k}\int_{g_t}^{T_{i+1}} \indi{Y^r_s-\underline Y^r_{[s,t]}\leq \eps_k}\, ds
\\&+
\liminf_{k\to\infty}\frac{1}{\eps_k}\int_{T_{i+1}}^{T_j} \indi{Y^r_s-\underline Y^r_{[s,t]}\leq \eps_k}\, ds
+
\liminf_{k\to\infty}\frac{1}{\eps_k}\int_{T_j}^{t} \indi{Y^r_s-\underline Y^r_{[s,t]}\leq \eps_k}\, ds\nonumber
\end{align}\end{linenomath}We now prove that almost surely, the first and last summands in both decompositions coincide and that the second and third summands in the decomposition of $\imf{H^\circ}{Y^r}_t$ are zero. 
\begin{description}
\item[First summand]
By construction, $g_t$ is upward for $Y^r$ and $g_t\leq T_j$. 
The induction hypothesis hence implies the equality $H_{g_t}=\imf{H^\circ}{Y^r}_{g_t}$. 
\item[Last summand] 
Note that $t-T_j$ is upward for $X^{j+1}\circ C^r$. 
Since $H^{j+1}$ is a continuous extension of $\imf{H^u}{X^{j+1}\circ C^r}$, we obtain
\begin{linenomath}\begin{align*}
H^{j+1}_{t-T_{j}}
&=\imf{H^\circ}{X^{j+1}\circ C^r}_{t-T_j}
\\&=\liminf_{k\to\infty}\frac{1}{\eps_k}\int_{0}^{t-T_j} \indi{X^{j+1}\circ C^r_s-\underline X^{j+1}\circ C^r_{[s,t]}\leq \eps_k}\, ds
\\&=\liminf_{k\to\infty}\frac{1}{\eps_k}\int_{T_j}^{t} \indi{Y^r_s-\underline Y^{r}_{[s,t]}\leq \eps_k}\, ds. 
\end{align*}\end{linenomath}
\item[Third summand] Note that $Y^r_s> \underline Y^r_{[T_j,t]}$ for any $s\in [T_{i+1},T_{j}]$. 
Hence, 
\begin{linenomath}\begin{esn}
\int_{T_{i+1}}^{T_{j}}\indi{\underline Y^r_u- \underline Y^r_{[u,t]}\leq \eps}\, du=0
\end{esn}\end{linenomath}for $\eps$ small enough. 
\item[Second summand] Note that $g_t$ is an upward time. 
Define also the upward time 
\begin{lesn}
g_t^k=\sup\set{s\leq T_{i+1}: Y^r_s\leq \eps_k+\underline Y^r_{[T_j,t]}},
\end{lesn}which decreases to $g_t$ as $k\to\infty$. 
Let $v$ be rational in $(g_{t}^1,T_{i+1})$ and such that $Y^r_{g_t^1-}\leq \underline Y^r_{[g_t^1,v]}$. 
Hence,  we also get $Y^r_{g_t^k-}\leq \underline Y^r_{[g_t^k,v]}$ for any $k$ as well as $Y^r_{g_t-}\leq \underline Y^r_{[g_t,v]}$. 
Note that\begin{lesn}
\int_{g_t}^{T_{i+1}}\indi{Y^r_s-\underline{Y^r}_{[s,t]}\leq \eps_k}\, ds
\leq \int_{0}^{g_t^k-g_t}\indi{X^{i+1}\circ C^{i+1}_s-\underline{X^{i+1}\circ C^{i+1}}_{[s,\infty)}\leq \eps_k}\, ds. 
\end{lesn}Thanks to 
Proposition \ref{timechangedHeightProcessForReflectedLevyProcessProposition} 
we get
\begin{lesn}
\liminf_{\eps_k\to\infty}\frac{1}{\eps_k}\int_{g_t}^{T_{i+1}}\indi{Y^r_s-\underline{Y^r}_{[s,t]}\leq \eps_k}\, ds
\leq \liminf_{k\to\infty} \imf{H}{X}\circ C_{g^k_t-g_t}=0. 
\end{lesn}
\end{description}Hence, $H$ is a continuous extension of $\imf{H^u}{Y^r}$. 
\end{proof}

Let us now turn to the construction of the supercritical L\'evy tree. 
For this, let $H^r$ denote the continuous modification of the height process of $Y^r$. 
We let $\tr{g}^r=\paren{\paren{\tau_r,d_r,\rho_r}, \leq_r,\mu_r}$ denote the TOM tree coded by $H^r$ and define $\zeta_r=\imf{\mu_r}{\tau_r}$. 
Let us see that $\tr{g}^r$ is a subtree of $\tr{g}^{r'}$ if $r\leq r'$. 
\begin{lemma}
\label{isometricEmbeddingBetweenPartsOfSupercriticalLevyTreeLemma}
If $r\leq r'$ then there exists an isometry $\fun{\iota}{\tau_r}{\tau_{r'}}$ 
such that
\begin{enumerate}
\item if $\sigma_1\leq_r \sigma_2$ then $\imf{\iota}{\sigma_1}\leq_{r'}\imf{\iota}{\sigma_2}$ and
\item the image of $\mu_{r'}$ under $\iota$ is the trace of $\mu_{r'}$ on $\iota(\tau_r)$. 
\end{enumerate}
\end{lemma}
\begin{proof}
For this proof, we denote by $A^{r',r}_t=\int_0^t \indi{Y^{r'}_s\leq r}\, ds$ and 
let $C^{r',r}$ be its right-continuous inverse. 
We will suppose that the time-change consistency of the $Y^r$ is valid pathwise, so that $Y^{r'}\circ C^{r',r}=Y^r$. 
Also, the proof of Proposition \ref{timechangedHeightProcessForReflectedLevyProcessProposition} allows us to see that: if $s$ is upward for $Y^r$ then $Y^r_{s-}<r$,  $C^{r',r}_s$ is upward for $Y^{r'}$, and 
\begin{lesn}
H^r=H^{r'}\circ C^{r',r}.
\end{lesn}We will also denote consider the set $[s]_r$ to be the equivalence class of $s$ under $\sim_{H^r}$. 
Note that $H^r$ is defined on $[0,\zeta_r]$.

To construct $\iota$, we will define $\tilde\iota$ on $[0,\zeta_r]$ by $\imf{\tilde\iota }{s}=C^{r',r}_s\in [0,\zeta_{r'}]$. 
Let $s_1<s_2$. 
Let us observe that\begin{equation}\label{timeChangedMinimumHeightProcessEquation}
\underline H^{r'}_{[C^{r',r}_{s_1},C^{r',r}_{s_2}]}=\underline H^r_{[s_1,s_2]}.
\end{equation}Indeed, note  first that on any interval of the form $I=[C^{r',r}_{s-},C^{r',r}_{s}]$ with $s\in [s_1,s_2]$,  
we have the inequality $H^{r'}_v\geq H^{r'}\circ {C^{r',r}_{s-}}$  for $v\in I$. 
We will prove it when $v$ is an upward time. 
Consider a rational $w\in (v,C^{r',r}_s)$ such that $Y^{r'}_{v-}\leq \underline Y^{r'}_{[v,w]}$. 
Since $Y^{r'}$ has an excursion above $r$ on $I$, 
we see that $Y^{r'}_{C^{r',r}_{s-}-}\leq \underline Y^{r'}_{[C^{r',r}_{s-},w]}$ and so \eqref{localTimeLemmaFromDLG} gives 
\begin{lesn}
H^{r'}_v=L^{w}_{w}-L^w_{w-v}\geq L^{w}_{w}-L^w_{w-C^{r',r}_{s-}}=H^{r'}_{C^{r',r}_{s-}}. 
\end{lesn}By continuity of the height process
  \begin{linenomath}\begin{lesn}H^{r'}\circ {C^{r',r}_{s-}}=H^{r}_s=H^{r'}\circ {C^{r',r}_{s}}\geq \underline H^r_{[s_1,s_2]}.
  \end{lesn}\end{linenomath}Hence, the equality $H^r=H^{r'}\circ C^{r',r}$ allows us to conclude the validity of equation \eqref{timeChangedMinimumHeightProcessEquation}. 

We now assert that if $[s_1]_r=[s_2]_r$ then $[C^{r',r}_{s_1}]_{r'}=[C^{r',r}_{s_2}]_{r'}$. 
  By hypothesis $H^r_{s_1}=H^r_{s_2}=\underline H^r_{[s_1,s_2]}$. 
  Hence, $H^{r'}\circ {C^{r',r}_{s_1}}=H^{r'}\circ {C^{r',r}_{s_2}}$ and by equation \eqref{timeChangedMinimumHeightProcessEquation}, 
  we see that\begin{linenomath}\begin{esn} \underline H^{r'}_{[C^{r',r}_{s_1},C^{r',r}_{s_2}]}=H^{r'}_{C^{r',r}_{s_1}}.
  \end{esn}\end{linenomath}We can then define\begin{linenomath}\begin{esn}
  \imf{\iota}{[s]_{r}}=[C^{r',r}_s]_{r'}.
  \end{esn}\end{linenomath}We have just proved that $\imf{C^{r',r}}{[s]_r}\subset [C^{r',r}_s]_{r'}$. 
  Although the converse inclusion might be false, we now see that nonetheless if $s_*=\sup [s]_r$ * (which belongs to $[s]_r$) then $C^{r',r}_{s_*}=\sup [C^{r,r'}_s]_{r'}$. 
  Indeed, we have just proved that $C^{r',r}_{s_*}\in \sup [C^{r,r'}_s]_{r'}$ and by hypothesis, for any $\eps>0$ we have $\underline H^{r'}_{[C^{r',r}_{s_*},C^{r',r}_{s_*+\eps}]}=\underline H^r_{[s_*,s_*+\eps]}<H^r_{s_*}$. 
  We conclude that $C^{r',r}_{s_*+\eps}\not\in [C^{r,r'}_s]_{r'}$ for any $\eps>0$, so that
  $C^{r',r}_{s_*}=\sup [C^{r,r'}_s]_{r'}$. 
  

To see that $\iota$ is an isometry, 
  note that if $s_1<s_2$ (say) then the distance of $[s_1]_{r}$ and $[s_2]_r$ equals, by equation \eqref{timeChangedMinimumHeightProcessEquation}:
  \begin{linenomath}\begin{esn}
  H^r_{s_1}+H^r_{s_2}-2\underline H^r_{[s_1,s_2]}
  =H^{r'}_{C^{r',r}_{s_1}}+H^{r'}_{C^{r',r}_{s_2}}-2\underline H^{r'}_{[C^{r',r}_{s_1},C^{r',r}_{s_2}]},
  \end{esn}\end{linenomath}and the right-hand side is the distance between $[C^{r',r}_{s_1}]_{r'}$ and $[C^{r',r}_{s_2}]_{r'}$. 

The order preserving character of $\iota$ is immediate since we have proved that $C^{r',r}_{\sup [s]_{r}}=\sup [C^{r,r'}_s]_{r'}$. Hence, if $s_1=\sup [s_1]_r\leq \sup [s_2]_{r}=s_2$ then
\begin{linenomath}\begin{esn}
\sup[C^{r',r}_{s_1}]_{r'}=C^{r',r}_{s_1}
\leq C^{r',r}_{s_2}= \sup[C^{r',r}_{s_2}]_{r'}. 
\end{esn}\end{linenomath}

Consider the image of Lebesgue measure on $[0,\zeta_r]$ under $C^{r',r}$. 
Since the inverse image of an interval $[0,t=$ under $C^{r',r}$ is $[0,A^{r',r}_t)$, we see that 
the image of Lebesgue measure on $[0,\zeta_r]$ under $C^{r',r}$ equals the measure induced by $A^{r',r}$. 
The latter is Lebesgue measure concentrated on $\set{t: Y^{r'}_s\leq r}$. 
By projecting to each of the trees coded by $Y^r$ and $Y^{r'}$ we see that 
$\imf{\mu_{r'}}{A\cap \imf{\iota}{\tau_r}}=\imf{\mu_{r}}{\imi{A}{\iota}}$. 
\end{proof}

Thanks to Lemma \ref{isometricEmbeddingBetweenPartsOfSupercriticalLevyTreeLemma}, and a direct limit argument used for the construction of locally compact TOM trees out of trees consistent under truncation, we deduce the existence of a locally compact TOM tree $\paren{\paren{\Gamma,d,\rho},\leq,\mu}$ and a growing sequence of subTOMtrees $\paren{\paren{\Gamma_r,d,\rho},\leq,\mu_r}$ (where $\mu_r$ is the restriction of $\mu$ to $\Gamma_r$) such that $\bigcup_r \Gamma_r=\Gamma$ and $\Gamma_r$ is isomorphic to the tree coded by $H^r$. 
The law of $\paren{\paren{\Gamma,d,\rho},\leq,\mu}$ will be denoted $\gamma^{\text{lc}}$. 
We also define $\gamma^c$ as the law of the tree coded by $H$ under $n^\#$ and finally set $\gamma=\gamma^{\text{c}}+b\gamma^{\text{lc}}$; 
for us $\gamma$ represents the \emph{law}  of supercritical L\'evy trees.

\section{Ray-Knight type theorems for supercritical L\'evy trees}
\label{rkSection}
We now pass to an interesting property of our supercritical L\'evy trees: 
their Ray-Knight theorem stated as Theorem \ref{RayKnightTheorem} and Corollary \ref{RayKnightCorollary}. 

To accomplish it, we will give a grafting description for the genealogical tree under $\Upsilon$. 
Then, the analysis will be extended under $\Upsilon_{\text{tree}}$. 
\subsection{A grafting construction for the genealogy under $\Upsilon$ and the corresponding Ray-Knight theorem}
Recall the construction of the TOM tree $S$ with law $\Upsilon$ as the pointwise direct limit of truncated trees $(S^n)$ coded by $(Y^n)$.
Formally, we have not defined the genealogy under $\Upsilon$, for which it suffices to follow the same path as under $\Upsilon_{\text{tree}}$: 
we define $H^n$ as a continuous modification the height process of $Y^n$, 
note that the tree coded by $H^n$, say $G^n$, 
is compatible under pruning, 
and define $G$ as the pointwise direct limit of the sequence $(G^n)$. 

For this, recall the processes $X^0,X^1,\ldots$ used to build $Y^n$ in the proof of Proposition \ref{uniquenessInfiniteLineOfDescentProposition}. 
Let $H^i$ be a continuous modification of the height process of $X^i$. 
We start by noting that $H^0$ has been analyzed in Lemma 8 of \cite{MR1883717}; to present the analysis (to be used) we first collect some preliminaries on $X^0$. 

For simplicity, we will now only consider the case when $\kappa=0$. 
First of all, the laws $\p_x^\rightarrow$, including the law $\p_0^\rightarrow$ of $X^0$, satisfy the following William's type decomposition, first extended to L\'evy processes in \cite{MR1419491} and further discussed in the spectrally positive case in \cite[Ch. 8]{MR2320889}.  For $x>0$, $\p_x^\rightarrow$ equals $\p_x$ conditioned on remaining positive (an event of positive probability). 
Under $\p_x$, the minimum of $X$ is achieved at a unique time and continuously, since $X$ is of infinite variation. 
(This fact was first proved in \cite{MR0433606} and can also be deduced from Proposition 1 and Theorem 1 in \cite{MR2978134}.) 
Let $T$ be the time the minimum is achieved and define the pre and post-minimum processes as $X^{\leftarrow}$ equal to $X$ killed at $T$ and $X^{\rightarrow}=X_{T+\cdot}-X_T$. 
Then, these two processes are independent (both under $\p_x$ as under $\p_x^\rightarrow$). 
(This is a classical and fundamental result of the fluctuation theory of L\'evy processes first found in \cite{MR588409}, which can also be deduced without local time considerations from Theorem 4 in \cite{MR2978134}.) 
Furthermore, under $\p_x$ and $\p_x^\rightarrow$, the law of $x-X_T$ is exponential of parameter $b$ (resp. exponential of parameter $b$ conditioned on being smaller than $x$) and, conditionally on $X_T=y\in (0,x)$, the law of $X^{\rightarrow}$ is $\p^\rightarrow$, while the law of $X^{\leftarrow}$ equals the image of $\q_{y-x}^{\#}$ under the mapping $f\mapsto f+x$. The law $\p^\rightarrow_{x,y}$ equal to $\p_x^\rightarrow$ conditioned on $\underline X_\infty=y$ just described give rise to a weakly continuous disintegration.

Let $\dunderline{X}^0$ be the future infimum process of $X^0$ given by
\begin{linenomath}
\begin{esn}
\dunderline{X}^0_t=\underline{X}^0_{[t,\infty)}=\inf_{s\geq t} X^0_s. 
\end{esn}\end{linenomath}Since our Laplace exponent is supercritical, then $\lim_{t\to\infty}\dunderline{X}^0_t=\infty$ and the set
\begin{linenomath}
\begin{esn}
\dunderline{\mc{Z}}=\set{t\geq 0: X^0_t=\dunderline{X}^0_t}
\end{esn}\end{linenomath}is unbounded while being regenerative. 
More specifically, from Lemma 8.(i) of \cite{MR1883717}, the process $X^0-\dunderline{X}^0$ is regenerative at zero and admits the following reconstruction by excursions.
Let $\dunderline{L}$ be the regenerative local time of $X^0_t-\dunderline{X}^0_t$ fixed by the normalization
\begin{linenomath}
\begin{esn}
\dunderline{L}_t=\lim_{\eps\to 0}\frac{1}{\eps}\int_0^t \indi{X^0_s-{\dunderline{X}}^0_s\leq \eps}\, ds. 
\end{esn}\end{linenomath}By recurrence, we see that $\dunderline{L}_\infty=\infty$. 
Let $\dunderline{\tau}$ be the right continuous inverse of $\dunderline{L}$. 
Then, with this normalization of the local time, the point process of excursions
\begin{linenomath}
\begin{equation}
\label{reflectedFutureMinimumPointProcess}
\sum_{s: \Delta\tau_s\neq 0}\delta_{(s,(X-\dunderline X)_{(\tau_{s-}+\cdot)\wedge\tau_s})}
\end{equation}\end{linenomath}is a Poisson point process on $(0,\infty)\times E$ with intensity
\begin{linenomath}
\begin{equation}
\label{LambertsIntensityForFutureMinimum}
\beta n^\#+\int_0^\infty e^{-bx}\imf{\overline \pi}{x}\, \q^\#_x\, dx.
\end{equation}
\end{linenomath}Note that integral equals the intensity of excursions that start at a positive value, 
corresponding to excursions above the future minimum which start with a jump. 
The excursions only record the jump of $X-\dunderline X$.
We will need a slightly more precise result which records also the jumps of the future minimum 
at the beginnings of excursion times, or equivalently, that records the jump of $X$. 
It can be guessed from the aforementioned Lemma 8.(i) in \cite{MR1883717}. 
\begin{proposition}
\label{characterizationOfExcursionsAboveFutureMinimumProposition}
Under $\p^\rightarrow_0$, 
the point process
\begin{linenomath}
\begin{equation}
\label{reflectedFutureMinimumPointProcessWithJumps}
\Xi^f=\sum_{s: \Delta\tau_s\neq 0}\delta_{(s,\Delta X_{\tau_s},(X-\dunderline X)_{(\tau_{s-}+\cdot)\wedge\tau_s})}
\end{equation}\end{linenomath}is a Poisson point process on $[0,\infty)\times E$ with intensity
\begin{lesn}
\imf{\mu^f}{dy,df}
= \imf{\delta_0}{dy}\beta \imf{n^\#}{df} 
+ 
\int_0^y e^{-bx}\imf{\q^\#_x}{df}\, dx\, \imf{\pi}{dy}.
\end{lesn}
\end{proposition}
The proof will be presented at the end of this subsection. 

Let $\dunderline{g_t}$ and $\dunderline{d_t}$ stand for the beginnings and ends of the excursions of $X^0$ above its future minimum process
\begin{linenomath}
\begin{esn}
\dunderline{g_t}=\sup\set{s\leq t: \dunderline{X}^0_t=X^0_t}
\quad\text{and}\quad
\dunderline{d_t}=\inf\set{s> t: \dunderline{X}^0_t=X^0_t}. 
\end{esn}\end{linenomath}
We also deduce, by the approximation result of \eqref{localTimeLemmaFromDLG} applied at time $\dunderline{g_t}$, 
that
\begin{linenomath}
\begin{equation}
\imf{H}{X^0}_t=\dunderline{L}_t+\lim_{k\to\infty}\frac{1}{\eps_k}\int_{\dunderline{g}_t}^t \indi{X^0_s-\underline{X^0}_{[s,t]}<\eps_k}\, ds\quad\text{and}\quad
\imf{H}{X^0}_{g_t}=\dunderline{L}_{g_t}=\dunderline{L}_{t}. 
\label{heightProcessOfPostMinimum}
\end{equation}\end{linenomath}In any case, we see that $\imf{H}{X^0}\geq \dunderline{L}$ 
(actually $\dunderline{L}$ is the future minimum process of $\imf{H}{X^0}$) and so $\imf{H}{X^0}_t\to \infty$ as $t\to\infty$. We can also describe the excursions of $\imf{H}{X^0}$ above its future minimum process: on an excursion interval $(g_t,d_t)$, we note that $\imf{H}{X^0}_t-L_t$ is a continuous extension of $\imf{H^u}{X^0_{g_t+\cdot}}_{\cdot-g_t}$; in other words, it is the image under the height process of the excursion of $X^0$ above $\dunderline{X^0}$. 
To finish the construction of $G$, 
let $C^{i,n}$ be the time-change that removes what is above $n$ from $X^i$, 
say defined on $[0, T^n_i-T^n_{i-1}]$ (with $T^n_0=0$). 
Then, define $H^n=\imf{H}{X^0}$ on $[0,T^n_1]$ and, recursively, for $t\in [T^n_i,T^n_{i+1}]$
\begin{linenomath}
\begin{esn}
g^n_t=\sup\set{s\leq T^n_i: Y^n_s\leq \underline Y^n_{[T^n_i,t]}}\quad\text\and H^n_t=H^n_{g^n_t}+\imf{H}{X^i}\circ C^{i,n}_{t-T^n_i}. 
\end{esn}\end{linenomath}Arguing as in the proof of Proposition \ref{timechangedHeightProcessForReflectedLevyProcessProposition}, we note that $H^n$ is a continuous extension of $\imf{H^u}{Y^n}$ and that the sequence of trees $(G^n)$ coded by $(H^n)$ is consistent under pruning, so that $G$ can be built as a pointwise direct limit $(G^n)$. 

Let $\q_x^\#$ be the law image of $\p_x^\#$ by killing upon reaching zero. 
\begin{proposition}
\label{PoissonDescriptionForGenealogyOfSINTreeProposition}
The tree $G$ is a sin tree. 
Let $\gamma^\#$ and $\gamma^\#_x$ be the \emph{laws} of the height process under $n^\#$ and $\q^\#_x$. 
Let $\Xi^1=\sum \delta_{(r_n^1,f^1_n)}$, $\Xi^2=\sum\delta_{(r_n^2,f^2_n)}$ and $\Xi=\sum\delta_{(r_n,f^l_n,f^r_n)}$ be Poisson point processes on $\excspace$, $\excspace$ and $\excspace^2$ with intensities $\nu^c,\nu^c$ and $\nu^d$ given by
\begin{linenomath}
\begin{align*}
\imf{\nu^c}{df}&= \beta \imf{\gamma^{\#}}{df}, 
\intertext{and}
\imf{\nu^d}{A\times B}&=\int e^{-b x}\indi{x\leq y} \imf{\gamma^\#_x}{A}\imf{\gamma^\#_{y-x}}{B} \, dx\, \imf{\pi}{dy}. 
\end{align*}
\end{linenomath}On the TOM tree $[0,\infty)$ rooted at zero, graft the trees coded by $f^1_n$ and $f^l_n$ to the left at heights $r^1_n$ and $r_n$ and graft the trees coded by $f^2_n$ and $f^r_n$ to the right at heights $r^2_n$ and $r_n$. 
The resulting TOM tree has the same law as $G$. 
\end{proposition}
\begin{proof}
The reader is asked to recall the proof of Proposition \ref{uniquenessInfiniteLineOfDescentProposition}. 
During that proof, we identified trees grafted to the left of the infinite line of descent of $S$ as excursions of $X^0$ above its future minimum process as well as trees grafted to the right as excursions above the cumulative minimum process of $X^1,X^2,\ldots$. 
The grafting heights are the heights in each $X^i$ at which the corresponding excursion ends. 
A similar analysis is valid for $G$ except that we use excursions of the height processes involved. 
Note first that the future minimum process of $\imf{H}{X^0}$ is $\imf{\dunderline{L}}{X^0}$ as follows from \eqref{heightProcessOfPostMinimum}.  
Since the left-hand side of the infinite line of descent $G$ can be coded by $\imf{H}{X^0}$, 
then trees grafted to the left of the infinite line of descent of $G$ are coded by excursions of $\imf{H}{X^0}$ above $\imf{\dunderline{L}}{X^0}$. 
Let $(g,d)$ be an excursion interval of $X^0$ above its future minimum process. 
Since upward times are dense 
(recall the discussion after \eqref{localTimeLemmaFromDLG}), 
and $g$ is one of them, 
we can use the approximation \eqref{localTimeLemmaFromDLG} 
at any rational $u>d$ and continuity of the height process and deduce that $\imf{H}{X^0}_g=\imf{H}{X^0}_d$. 
Now the analysis breaks down into two cases: 
when $X^0_{g-}=X^0_g$ 
(or in other words, when the excursion starts continuously for $X^0$) or when $X^0_{g-}<X^0_g$. 
In the former case, note that $\imf{H}{X^0}>\imf{H}{X^0}_g$ on $(g,d)$ 
(by support properties of local times) 
so that $H$  on $[g,d]$ codes  a subtree grafted to the left of the infinite line of descent 
and $d$ and $g$ correspond in $G$ to a binary branchpoint 
(upon removal, it disconnects the tree into $3$ components), 
while in the latter case, 
we have $\imf{H}{X^0}_t=\imf{H}{X^0}_g$ for $t\in (g,d)$ if and only if $X^0_{t-}= \underline X^0_{[g,t]}$. 
Note then that all such $t$ correspond to the same point on $G$ and the (sub)excursion interval 
codes a tree grafted to the left of the infinite line of descent. 
Hence, the element of $G$ corresponding to them is an infinite branch point 
(upon removal, it disconnects the tree into an infinite number of components). 
To find subtrees to the right of the infinite line of descent, 
the analysis is also divided between those corresponding to infinite branch points 
and those corresponding to binary branch points. 
The former are constructed as follows: 
consider the (vertical) interval $I=(\dunderline{X}^0_{g-},\dunderline{X}^0_{d})\cap [n-1,n]$ for $n\geq 1$. 
If $(g',d')$ is an excursion of $X^n$ above its cumulative minimum and $X^n_{d'}\in I$ then, 
by definition, 
$H^n_{T^m_n+t}=H_{g-}=\dunderline{L}_{g-}$ for all $m\geq n$ 
and $t\in [g',d']$ such that $X^n_{t-}=\underline X^n_{t}$, in particular $g'$ or $d'$. 
Again, the element of $G$ corresponding to all such $t$ is an infinite branch point. 
The binary branch points are constructed from excursions of $X^i$ 
above their cumulative minimum process, 
say on the excursion interval $(g',d')$ where $X^n_{d'}$ 
does not belong to the jump intervals $(\dunderline{X}^0_{g-},\dunderline{X}^0_d)$. 

Let us now see at which heights the compact trees are grafted to the left of the infinite line of descent of $G$. 
Since the height along the infinite line of descent equals the local time of $X^0-\dunderline{X^0}$, (since $H=L$ at ends of excursions), 
then an excursion of $X^0$ on $[g,d]$ gives rise to a tree grafted to the left of the infinite line of descent of $G$ at height $L_g$. 
If $X^0_g>X^0_{g-}$,  then we must graft a tree at the same height at the right of the infinite line of descent; the tree is coded, using the same notation as before, by $H^n_{T^m_n+\cdot}$ on $[g',d']$ for large enough $m$. 
We see that the left of the infinite line of descent can be given a Poissonian construction as follows, 
thanks to Proposition \ref{characterizationOfExcursionsAboveFutureMinimumProposition}: 
along $[0,\infty)$ (viewed as a vertical locally compact TOM tree), 
graft trees to the left with intensity 
$\beta\gamma^\#+\int_0^\infty e^{-bx}\imf{\overline \pi}{x}\, \gamma^\#_x\, dx$. 
The intensity with density $x\mapsto e^{-bx}\imf{\overline \pi}{x}$ corresponds to the sizes of overshoots 
$\Delta \paren{X^0-\dunderline{X}^0}$ above the future minimum process. 
However, to the trees with law $\gamma_x^\#$, 
which correspond to the overshoot of $X^0$ when the future minimum jumps, 
we must add the corresponding trees to the right of the infinite line of descent but at the same height. 
If we want to capture not only the overshoot but also the complete  size of the jump $\Delta X^0$ 
at each jump over the future minimum, 
then the intensity becomes $(x,y)\mapsto e^{-b x}\indi{x\leq y}  \, dx\, \imf{\pi}{dy}$ thanks to Proposition \ref{characterizationOfExcursionsAboveFutureMinimumProposition}. 
With the trees that get grafted, we obtain the intensity $\nu^d$ of the statement. 
Finally, to the right of the infinite line of descent we also have trees 
which come from the continuous excursions of the $X^i$ ($i\geq 1$) 
above its cumulative minimum processes. 
In their natural local time scale, these arrive at rate $n^\#$. 
We now prove that in the time scale of $\dunderline{L}$, 
the intensity is actually $\beta \gamma^\#$, which concludes the proof of the theorem. 
Let $\dunderline{\tau}$ be the right-continuous inverse of $\dunderline{L}$. 
We recall that binary branch points along the right of the infinite line of descent are coded by $\imf{H}{X^n}$  on excursion intervals $(g,d)$ of $X^n-\underline X^n$ 
where $X^n_d$ belongs to the range of $\dunderline{X}^0$. 
To examine the latter, recall that Lemme 4 in \cite{MR1141246} tells us 
that the joint law of $(X^0-\dunderline X^0, \dunderline X^0)$ is the same as that of 
$(\mc{R}(\overline X-X), \overline X\circ d)$ under $\p$, 
where $\mc{R}(\overline X-X)$ is a process obtained by reversing the arrow of time 
of each excursion of $\overline X-X$ and $d_t$ is the right endpoint 
of the excursion of $\overline X-X$ straddling time $t$. 
However, local times in the Duquesne-Le Gall normalization of \eqref{DuquesneLeGallNormalization} 
are invariant under time-reversal, 
so that the joint law of $(\dunderline{L}, \dunderline{X}^0)$ coincides with that of $(L,X\circ d)$ under $\p$. 
Finally, noting that composition is measurable as in \cite{MR561155} or \cite{MR1876437} 
and using the equality $d\circ L^{-1}=L^{-1}$, we see that the law of $\dunderline{L}$ 
is the same as that of the ladder height process $X\circ L^{-1}$. 
Lemma 1.1.2 in \cite{MR1954248} tells us that $X\circ L^{-1}$ has drift coefficient $\beta$ 
(in the particular normalization of local time), and so Proposition 1.8 of \cite[p.13]{MR1746300} 
tells us that
\begin{linenomath}
\begin{esn}
\beta t=\imf{\leb}{\set{ \dunderline{X}_s: s\geq 0}\cap [0,\dunderline{\tau}_t]}. 
\end{esn}\end{linenomath}Hence, we deduce that trees to the right of the infinite line of descent of $G$ 
that are rooted at binary branch points are still a Poisson point process with intensity $\beta \gamma^\#$. 
\end{proof}
We now turn to the Ray-Knight theorem associated to the tree $G$. 
Recall that $G$ is the genealogical tree associated to the tree $S$ with law $\Upsilon$. 
Recall also that $\Upsilon$ was constructed out of $\Psi$, that $\beta$ is the Gaussian coefficient in $\Psi$, 
$\pi$ is its L\'evy measure and $b$ is its greatest root. 
The reader may consult \cite{MR2485021} 
for Ray-Knight type theorems of sin trees featuring more general CBI processes. 
\begin{proposition}
Suppose that $G=\paren{(\tau,d,\rho),\leq,\mu}$ and let $\imf{\delta}{\sigma}=\imf{d}{\rho,\sigma}$. 
Then, the random measure $\Xi=\mu\circ\delta^{-1}$ 
admits a \cadlag\ density $Z$. 
The process $Z$ is a $\cbi$ process with subcritical branching mechanism $\Psi^\#$ 
and immigration mechanism $\Phi$ given by\begin{lesn}
\imf{\Phi}{\lambda}
=2\beta\lambda+ \int_0^\infty (1-e^{-\lambda x})\frac{1-e^{-bx}}{b}\, \imf{\pi}{dx}
=\frac{\imf{\Psi}{\lambda+b}-\imf{\Psi}{\lambda}}{b}
. 
\end{lesn}
\end{proposition}
The main tools in the proof are the Ray-Knight theorems under $\eta^{\#}$ 
as well  the spinal decomposition of $\cbi$ processes that we now briefly recall. 
For the details of spinal depompositions, 
the reader can consult \cite[Sect. 2.4]{2012arXiv1202.3223L} in full generality 
or \cite{MR2862645} under Grey's condition,  
as well as the streamlined exposition in \cite[Sect. 4]{MR3263091}.  
For details regarding the Ray-Knight theorems, we refer the reader to \cite{MR1954248} and \cite{MR2147221}. 
With $\Psi^\#$ and $\Psi$ as in the statement, 
let $P_x$ be the law of a $\imf{\cb}{\Psi^\#}$ 
(continuous-state branching process with branching mechanism $\Psi^\#$) that starts at $x$. 
It is then known that there exists a measure $Q$ 
(the Kuznetsov measure of $\Psi^\#$) on $\excspace$ 
such that if $\Xi=\sum_{n}\delta_{(t_n,f_n)}$ is a Poisson point process 
with intensity $\beta Q+\int_0^\infty \frac{1-e^{-bx}}{b} P_x\,\imf{\pi}{dx}$ 
then the process $Z$ given by $Z_t=\sum_{t_n\leq t} \imf{f_n}{t-t_n}$ 
is a $\imf{\cbi}{\Psi^\#,\Phi}$ 
(a continuous state branching process with immigration 
with branching mechanism $\Psi^\#$ and immigration mechanism $\Phi$). 
The law $Q$, called the Kuznetsov measure of $P_x$, is Markovian and admits same semigroup as $P_x$. 
On the other hand, the Ray-Knight theorem states that under $n^\#$ or under $\q_x^\#$, 
the random measure $A\mapsto \imf{\leb}{\set{t\in (0,\zeta)}: H_t\in A}$ 
admits a \cadlag\ density $Z$ which has law $Q$ or $P_x$. 
For the case of $\q_x^\#$, this is the content of Theorem 1.4.1 in \cite{MR1954248}. 
We were unable to find the case of $n^\#$ reported in the literature. 
However, a quick proof of it can be given by the fact that 
$\imf{n^\#}{1-e^{-\lambda t}}=-\log P_1 (e^{-\lambda X_t})$ 
(by the proof of Theorem 1.4.1 in \cite{MR1954248}) 
and this equals $\imf{Q}{1-e^{-\lambda X_t}}$ (as in equation (2) in \cite{MR2862645}). 
On the other hand, both measures are Markovian and have the same semigroup as $(P_x)$; in the case of $Q$ this follows by equation (2) in \cite{MR2862645} while for $n^\#$, this follows from the regenerative property  (of the tree coded by $H$) and the Ray-Knight theorem under $\q_x^\#$. 
\begin{proof}
Let $L$ be the unique infinite line of descent of $G$. 
We first show that $\imf{\mu}{L}=0$. 
	Indeed, consider first $X^0$ and its future infimum process $\dunderline{X}^0$. 
	Since the set
	\begin{linenomath}
	\begin{esn}
	\mc{L}^0=\set{t\geq 0: X^0_t=\dunderline{X}^0_t}
	\end{esn}\end{linenomath}has the same law as 
	$\set{t\geq 0: X_t=\overline X_t}$ under $\p$, 
	as recalled in the proof of Proposition \ref{PoissonDescriptionForGenealogyOfSINTreeProposition}, 
	we see that $\imf{\leb}{\mc{L}^0}=0$ since, 
	as noted in the proof of Proposition 7 in \cite{splitting1}, 
	the upward ladder time process under $\p$ has zero drift. 
	On the other hand, for each $n\geq 1$, the sets
	\begin{linenomath}
	\begin{esn}
	\mc{L}^n=\set{t\geq 0: X^n_t=\underline X^n_t}
	\end{esn}have measure zero whenever 
	the inverse of $\underline X$ under $\p$ has zero drift. 
	Since the Laplace exponent of the latter
	equals to the right continuous inverse of $\Psi$, 
	we see that it has no drift whenever $\beta>0$ 
	since this implies $\imf{\Psi}{\lambda}\sim \beta\lambda^2$ as $\lambda\to\infty$. 
	Finally, when $\beta=0$, the set 
	$\mc{R}=\set{\dunderline{X}^0\circ \dunderline{L}_t:t\geq 0}$
	has zero Lebesgue measure, as shown in the proof of Proposition \ref{PoissonDescriptionForGenealogyOfSINTreeProposition}. 
	Hence, 
	the set $\set{t\geq 0: X^n_t=\underline X^n_t\in \mc{R}}$ has zero Lebesgue measure. 
	Under the mapping sending $t$ to its equivalence under $\sim_{Y^n}$, 
	the sets we have considered are projected into the infinite line of descent, 
	which therefore has zero measure under $\mu$. 
	\end{linenomath}

Suppose that the compact trees grafted to the left and to the right of the infinite line of descent are enumerated as $(t_i,\tau_i)$, where $t_i$ is the distance from $\rho$ to the root $\rho_i$ of $\tau_i$. 
Then
\begin{lesn}
\imf{\rkmes}{A}
=\imf{\mu}{\set{\sigma\in L:\imf{d}{\rho,\sigma}\in A}}
+\sum_i \imf{\mu}{\set{\sigma\in \tau_i:\imf{d}{\rho_i,\sigma}\in A-t_i}}.
\end{lesn}As we have just seen, the first summand is zero. 
For the second sum, call each summand $\imf{\rkmes_i}{A}$. 
Thanks to the Ray-Knight theorem under $n^\#$ and under $\q_x$, 
let $Z^i$ be a density for the measure $\rkmes_i$, which has the semigroup of a  $\cb$ processes with branching mechanism $\Psi^\#$. 
Then
\begin{lesn}
\imf{\rkmes_i}{A}=\int_{A} Z^i_{t-t_i}\, dt
\end{lesn}and so $\Xi$ is absolutely continuous with respect to Lebesgue measure and a version of its density is $Z=\sum_i Z^i_{\cdot-t_i}$. 
Note that this is independent of the side of the infinite line of descent to which the trees $\tau_i$ are grafted. 
Since the concatenation of two processes with laws $\gamma_x^\#$ and $\gamma_y^\#$ has law $\gamma_{x+y}^\#$, then the image of $\nu$ in Proposition \ref{PoissonDescriptionForGenealogyOfSINTreeProposition} under the concatenation of both trajectories equals $\int_0^\infty (1-e^{-by})/b\, \gamma^\#_y\, \imf{\pi}{dy}$.
Thanks to the spine representation and the Poisson construction of $G$ in Proposition \ref{PoissonDescriptionForGenealogyOfSINTreeProposition}, 
we see that $Z$ is a $\cbi$ process with branching mechanism $\Psi^\#$ and the immigration mechanism $\Phi$ as stated. 
The stated relationship between $\Psi$ and $\Phi$ can be checked by computation. 
\end{proof}

A further consequence of the spinal decomposition of $\imf{\cbi}{\Psi^\#,\Phi}$ started at zero is the following.
At any time $t>0$, the post-$t$ evolution is decomposed into two parts: 
that corresponding to $f_n(t-t_n)$ for $t_n\leq t$ and that corresponding to what attaches to the spine above $t$. 
If $Z_t=x$, then the first part evolves as $\imf{\cb}{\Psi^\#}$ started at $x$ (thanks to the Markovian character of $Q$ and the branching property). Furthermore, the second contribution evolves as an independent $\imf{\cbi}{\Psi^\#, \Phi}$ started at zero. 
This remark be important to the proof of Theorem \ref{RayKnightTheorem}. 

We finally present the pending proof of this subsection. 

\begin{proof}[Proof of Proposition \ref{characterizationOfExcursionsAboveFutureMinimumProposition}]
We first comment on the regenerative character of $\mc{Z}$ 
in a way that handles the jump of $X$ when $\dunderline{X}$ also jumps. 
First, consider $\G_t=\F^{X,\dunderline{X}}_t$ 
and note that it coincides with $\F^X_t\vee \sag{\dunderline X_t}$ 
due to the equality $\dunderline{X}_s=\dunderline{X}_t\wedge \underline{X}_{[s,t]}$ valid whenever $s\leq t$. 
We first assert that $(X,\dunderline{X})$ is a Markov process under any $\p_x^{\rightarrow}$. 
Indeed, note that for any $A\in\F^X_t$, the Markov property and the definition of $\p_{x,y}^\rightarrow$ give
\begin{linenomath}
\begin{align*}
\imf{\se^{\rightarrow}_x}{\indi{A, \dunderline{X}_t\in B}\imf{g}{X_{t+s},\dunderline X_{t+s}}}
&=\imf{\se^{\rightarrow}_x}{\indi{A, \dunderline{X}_t\in B} \imf{h}{X_t,\dunderline{X}_t}}
\intertext{where}
\imf{h}{x,y}&=\imf{\se_{x,y}^{\rightarrow}}{\imf{g}{X_s,\dunderline{X}_s}}. 
\end{align*}\end{linenomath}This gives us the Markovian character of $(X,\dunderline X)$; 
by weak continuity of $\p_{x,y}$, it is even a Feller process. 
We now assert that $X-\dunderline{X}$ is Markovian with respect to the filtration $(\G_t)$. 
Indeed, note that the image of $\p^{\rightarrow}_{x,y}$ under the mapping 
$f\mapsto f-y$ is $\p^\rightarrow_{x-y,0}$, 
as follows from its definition and the spatial homogeneity of L\'evy processes. 
It follows that\begin{lesn}
\imf{\se_{x,y}^\rightarrow}{\imf{g}{X_s-\dunderline X_s}}
=\imf{\se_{x-y,0}^\rightarrow}{\imf{g}{X_s-\dunderline X_s}}, 
\end{lesn}so that
\begin{lesn}
\imf{\se^{\rightarrow}_x}{\indi{A,\dunderline{X}_t\in B}\imf{g}{X_{t+s}-\dunderline X_{t+s}}}
=\imf{\se^{\rightarrow}_x}{\indi{A,\dunderline{X}_t\in B}\imf{h}{X_t-\dunderline X_t}}
\end{lesn}with $\imf{h}{x-y}=\imf{\se_{x-y,0}^{\rightarrow}}{\imf{g}{X_s-\dunderline{X}_s}}$. 
Hence, $X-\dunderline{X}$ is Markovian (and indeed Feller) under $\p_x^\rightarrow$ with respect to the filtration $(\G_t)$. 
In conclusion, $\mc{Z}$ is regenerative with respect to the filtration $(\G_t)$. 
It follows that $\Xi^f$ is a Poisson point process. 
Indeed, let $\eps>0$ and, starting with $T_0=d_{0}=0$, let $T_{n+1}=\inf\set{t\geq d_{n}: X_t-\dunderline X_t\geq \eps}$,  $d_{n+1}=\inf\set{t\geq T_{n+1}:t\in \mc{Z}}$ and $g_{n+1}=\sup\set{t\leq T_{n+1}: t\in \mc{Z}}$. 
Then $T_n$ and $d_{n}$ are stopping times with respect to the filtration $(\G_t)$. 
Because of the strong Markov property applied at times $d_{n}$, we see that $\paren{X-\dunderline{X}}_{T_{n+1}+\cdot}$ is independent of $\G_{d_{n}}$ and in particular of $Y^n={X-\dunderline X}_{\paren{g_{m}+\cdot}\wedge d_{m}}$ and of $\Delta X_{g_{m}-}$ for any $m\leq n$.  
Therefore, the sequence $(\Delta X_{g_{T_m}-}, Y^n)$ is iid. 
When varying $\eps$, the sequences $(\Delta X_{g_{T_m}-}, Y^n)$ conform a nested array 
as introduced in \cite{MR579823}; 
the main result in that paper allows us to deduce that $\Xi^f$ is a Poisson point process, whose intensity we now compute. Note, however, that we can write its intensity $\tilde \mu^f$ as $\imf{\delta_0}{dy}\imf{\tilde n}{df}+\tilde \mu^{f,d}$, where $\tilde \mu^{f,d}$ is the restriction of $\tilde \mu^f$ to excursions which start with at a non-zero value.


We first need the following fact. 
Almost surely: 
a time $t>0$ is the beginning of a discontinuous excursion of $X-\dunderline X$ 
if and only if 
$t$ is a time of a common jump of $X$ and $\dunderline X$. 
At any such time $t$, 
we have the inequalities $\Delta X_t>\Delta \dunderline X_t>0$. 
Indeed, if $t$ is the beginning of a discontinuous excursion, 
then by definition we get that $0<\Delta (X-\dunderline X)=\Delta X_t-\Delta\dunderline{X}_t$ 
and $X_{t-}-\dunderline X_{t-}=0$. 
Since $\dunderline{X}$ is non-decreasing, 
then $\Delta X_t>0$. 
However, by ennumerating jumps of $X$ of size $>\eps$ for any $\eps>0$ 
and applying the strong Markov property and the absolute continuity of the law of the minimum of $X$, 
we see that $X_{t-}\neq \dunderline X_t$ and $\dunderline X_t<X_t$ at any jump time of $X$. 
Hence, we deduce that $X_{t-}=\dunderline X_t<\dunderline X_t<X_t$ which implies that $t$ is a common jump of $X$ and $\dunderline X$ and indeed the inequalities $0<\Delta{\dunderline X}_t<\Delta X_t$. 
On the other hand, if $t$ is a common jump of $X$ and $\dunderline X$ then $\dunderline X_{t}>\dunderline X_{t-}$ which implies that $X_{t-}=\dunderline X_{t-}$, so that $\paren{X-\dunderline{X}}_{t-}=0$. 
As we have remarked, since $t$ is a jump time of $X$ we then get the inequalities $X_{t-}<\dunderline{X}_t<X_t$ so that $t$ is a jump time of $X-\dunderline{X}$ and the beginning of a discontinuous excursion.  We have also obtained the inequality $\Delta{\dunderline X}_t<\Delta X_t$.

We will now construct a nested array of discontinuous excursions. 
For any $\eps>0$, 
let $T_n$ be the time of the $n$-th jump jump of $X$ of size greater than $\eps$ 
that is common to $X$ and  $\dunderline X$ and let $\rho_n=\inf\set{t\geq T_n: X_t=\dunderline{X}_t}$. 
Note that both $T_n$ and $\rho_n$ are stopping times with respect to the filtration $\paren{\G_t}$. 
Define\begin{lesn}
V_n=\imf{\Delta X}{T_n},\quad O_n=X_{T_n}-\dunderline{X}_{T_n}\quad\text{and}\quad F_n=X_{(T_n+\cdot)\wedge \rho_n}-\dunderline{X}_{T_n}. 
\end{lesn}Note that $O_n=F_n(0)$. 
Because of the strong Markov property, 
the random variables $\set{(V_n,O_n,F_n)}$ are independent and identically distributed, with a law depending on $\eps$. Note that as we vary $\eps$, we get a nested array that exhausts the discontinuous excursions of $X$ by the preceeding paragraph. 
The main theorem in \cite{MR579823} implies the existence of a $\sigma$-finite measure $\tilde \nu^f$ such that the law of $(V_1,O_1,F_1)$ is $\tilde \nu^f$ conditioned on $\re_+\times (\eps,\infty)\times E$. 
In fact, all measures satisfying this conditional property differ by a constant factor. 
Also, the conditional law of $F_1$ given $(V_1,O_1)=(y,x)$ is the image of $\q_{y-x}^\#$ under the mapping $f\mapsto x+f$. 
We now compute the law of $(V_1,O_1)$ and show that
\begin{lesn}
\proba{V_1\in dy, O_1\in dx, F_1\in df}=\frac{\indi{\eps\leq y} be^{-bx}\, \imf{\q^{\#}_x}{df}\, dx\, \imf{\pi}{dy}}{\int_\eps^\infty (1-e^{-by})\, \imf{\pi}{dy}}.
\end{lesn}If we let  $\tilde \mu^f$ be the image of $\tilde \nu^f$ by the map $(y,x,f)\mapsto (y,f)$, then our construction implies the existence of a constant $c$ such that $\tilde \mu^f$ equals $c \mu^f$ on discontinuous excursions and we will then argue that $c=1$.  

Let us compute the law of $(V_1,O_1)$. 
Since $\p^\rightarrow$ is the law of the post-minimum process under $\p$, it suffices to do the computation using the latter law. 
Let $S_1,S_2,\ldots$ be the succesive jumps of $X$ of size greater than $\eps$ and consider two Borel sets $B_1,B_2$ of $(\eps,\infty)$ and $(0,\infty)$. 
Then, 
using the Strong Markov property, 
the law of the overall minimum under $\p_x$,  
as well as the master formula of Poisson point processes, 
we obtain
\begin{linenomath}
\begin{align*}
&\proba{\Delta X_{T_1}\in B_1, X_{T_1}- \dunderline X_{T_1}\in B_2}
\\&=\sum_k\proba{S_k=T_1,\Delta X_{S_k}\in B_1, X_{S_k}- \dunderline X_{S_k}\in B_2\cap (0,\Delta X_{S_k})}
\\&=\sum_k\esp{\indi{\underline X_{[S_i,S_k)}\leq X_{S_i-},i<k}\indi{\Delta X_{S_k}\in B_1} 
\int_{B_2\cap (0,\Delta X_{S_k})} be^{-bx}\, dx
}
\\&=\esp{\int_0^\infty  \indi{\underline X_{[S_i,l)}\leq X_{S_i-}\text{ if }S_i<l}  \, dl}\int_{B_1} \int_{B_2\cap (0,y)} be^{-bx}\, dx \, \imf{\pi_\eps}{dy}.
\end{align*}\end{linenomath}In particular, taking $B_1=(\eps,\infty)$ and $B_2=(0,\infty)$, we see that
\begin{lesn}
\proba{\Delta X_{T_1}\in dy, X_{T_1}- \dunderline X_{T_1}\in dx}
=\frac{be^{-by}}{\int_\eps^\infty (1-e^{-by})\, \imf{\pi}{dy}}\indi{\eps,x\leq y}\, dx\, \imf{\pi}{dy}. 
\end{lesn}

Finally, as mentioned before, we have shown that $\Xi^f$ is a Poisson point process with intensity $\imf{\delta_0}{dy}\imf{\tilde n}{df}+c \int_0^y e^{-bx}\imf{\q_x^\#}{df}\, dx \imf{\pi}{dy}$. 
Note that the point process in \eqref{reflectedFutureMinimumPointProcess} is the image of the point process in \eqref{reflectedFutureMinimumPointProcessWithJumps} under the mapping $(s,y,f)\mapsto (s,f)$, which shows that $\tilde n=\beta n^\#$ and that $c=1$, so that the intensity of $\Xi^f$ is precisely $\mu^f$. 
\end{proof}
\begin{remark}
The proof in \cite{MR1883717} that allows us to conclude that $c=1$ depends on the theory of scale functions. 
A more simple argument would be to substitute the proof of Lemma 9 in \cite{MR1883717} for the proof of Lemma 1.2.1 in \cite{MR1954248}. 
Furthermore, elementary computations as the ones we used to compute the law of $(V_1,O_1)$ allow us to conclude that the post minimum process under $\p_x^\rightarrow$ has law $\p^\rightarrow$, thereby making the above arguments more self-contained. 
\end{remark}

\subsection{A grafting construction for the genealogy under $\Upsilon_{\text{tree}}$ and the corresponding Ray-Knight theorem}
In this subsection, we present the proof of Theorem \ref{RayKnightTheorem} and Corollary \ref{RayKnightCorollary}. 
Recall that $\gamma^{\text{lc}}$ stands for the limit of trees coded by the height process $H$ under the image of $\Upsilon_{\text{tree}}$ under truncation at height $r$ as $r\to\infty$. 
The strategy will be similar: we first prove a Poisson description of $\gamma^{\text{lc}}$. 
The difference will be that the Poisson description will only be recursive. 
We then use this Poisson description, 
  as well as the known Ray-Knight theorems under $n^{\#}$ and $\q^{\#}_x$, 
to conclude. 
To do this, we will need the notion of concatenation of trees, which is a particular form of grafting, performed at the root. 
First, if $f$ and $g$ are excursions in $\excspace$, we define their concatenation $f\vee g$ by
\begin{lesn}
\imf{f\vee g}{t}=\begin{cases}
f(t)& 0\leq t<\imf{\zeta}{f}\\
g(t-\imf{\zeta}{f})& \imf{\zeta}{f}\leq t<\imf{\zeta}{f}+\imf{\zeta}{g}\\
\dagger& t\geq \imf{\zeta}{f}+\imf{\zeta}{g}
\end{cases}. 
\end{lesn}Then, if $\tr{c}_1$ and $\tr{c}_2$ are two TOM trees coded by $f_1$ and $f_2$, we define $\tr{c}_1\vee \tr{c}_2$ as the tree coded by $f_1\vee f_2$. 
Finally, if  $\tr{c}_1$ and $\tr{c}_2$ are two locally compact TOM trees with coding sequence $(f_1^n)$ and $(f_2^n)$ we let $\tr{c}_1\vee \tr{c}_2$ have coding sequence $(f_1^n\vee f_2^n)$. The image of the product measure $\gamma_1\times \gamma_2$ on locally compact TOM trees under concatenation is denoted $\gamma_1\vee \gamma_2$. 
In our Poissonian description, we will use the measure $\gamma^k_z$ given by
\begin{lesn}
\gamma^k_z
=\int_{0=z_0<z_1<\cdots <z_k<z_{k+1}=z} \gamma^\#_{z_1-z_0}\vee \gamma^{\text{lc}}\vee \gamma^\#_{z_2-z_1}\vee \gamma^{\text{lc}} \vee \cdots \vee\gamma^{\text{lc}}\vee  \gamma^\#_{z_{k+1}-z_k} \, dz_1\,\cdots\, dz_k. 
\end{lesn}This measure corresponds to the intertwinning of $k+1$ compact trees and $k$ locally compact trees and uses the measure $\gamma^{\text{lc}}$ in its definition. 
\begin{proposition}
\label{PoissonDescriptionSuperCriticalLevyTreeProposition}
Let $\Xi^1=\sum \delta_{(r^1_n,f^1_n)}$, $\Xi^2=\sum \delta_{(r^2_n,f^2_n)}$ and $\Xi^3=\sum \delta_{(r_n,f^l_n,f^r_n)}$ be Poisson Point processes with intensities $\nu^c$, $\nu^c+\beta b \gamma^{\text{lc}}$ and $\nu^d$ where
\begin{linenomath}
\begin{align*}
\nu^c&=\beta \gamma^\#
\intertext{and}
\imf{\nu^d}{A\times B}&=
  \sum_{k=0}^\infty
    \int 
    b^k e^{-by}\indi{x\leq y}\imf{\gamma^\#_x}{A} 
    \imf{\gamma^k_{y-x}}{B} \, dx\, \imf{\pi}{dy}. 
\end{align*}
\end{linenomath}On the TOM tree $[0,\infty)$ rooted at zero, graft the trees coded by $f^1_n$ and $f^l_n$ to the left at heights $r^1_n$ and $r_n$ and graft the trees coded by $f^2_n$ and $f^r_n$ to the right at heights $r^2_n$ and $r_n$. 
The resulting TOM tree has law $\gamma^{\text{lc}}$.
\end{proposition}
\begin{proof}
Let us recall that $\Upsilon_{\text{tree}}$ is obtained from $\Upsilon$ by grafting, 
to the right of the unique infinite line of descent, 
iid trees with law $\Upsilon_{\text{tree}}$ at rate $b$. 
Specifically, if $S_\emptyset$ has law $\Upsilon$ and $S^*_1,S^*_2,\ldots$ are iid with law $\Upsilon^{\text{tree}}$ and we graft $S_i^*$ to the right of $S_\emptyset$ at height $T_i$, where $(T_i-T_{i-1})$ are iid exponentials with rate $b$, to obtain $S^*$, then $S^*$ has law $\Upsilon^{\text{tree}}$. 
We now use the Poisson description of the genealogical tree associated to $S_\emptyset$ stated as Proposition \ref{PoissonDescriptionForGenealogyOfSINTreeProposition} and use very similar arguments to prove the present proposition; 
only the differences will be explained. 
Let $G_\emptyset$ and $G^*_1,G^*_2,\ldots$ be the genealogical trees associated to $S_\emptyset$ and $S^*_1,S^*_2,\ldots$. 
We have already identified the parts of the tree $S_\emptyset$ that give rise to the Poisson description of $G_\emptyset$. 
It only remains to see how $G^*_1,G^*_2,\ldots$ are grafted to the right of the unique infinite line of descent of $G_\emptyset$. 
For this, suppose that the left of the infinite line of descent of $S_\emptyset$ is coded by $X^{\emptyset,0}$ (which has law $\p^\rightarrow$). Recall that the infinite line of descent of $S_\emptyset$ was identified with the heights $\dunderline{X}^{\emptyset, 0}_t$ corresponding to $t$ such that $X^{\emptyset, 0}_t=\dunderline{X}^{\emptyset, 0}_t$. 
These heights leave open gaps corresponding to the jumps of $\dunderline{X}^{\emptyset,0}$ and anything grafted on these gaps gets contracted to the same point when considering the genealogy. However, additionally to what is grafted on these gaps to form $S_\emptyset$, we now graft independently the trees $(S_i)$ at rate $b$. 
More formally, suppose that $\Delta\dunderline{X}^{\emptyset,0}_t>0$, where $X^{\emptyset,0}_t-\dunderline{X}^{\emptyset,0}_{t}=x$, $X^{\emptyset,0}_t-\dunderline{X}^{\emptyset,0}_{t-}=y$ (so that $x\leq y$) and where the minimum of $X^{\emptyset,0}$ on $[t,\infty)$ is reached at $\dunderline{d}_t$. 
Then, the quantity $K$ of trees $(S_i)$ that get grafted to the right of the gap $(\dunderline{X}^{0,\emptyset}_{t-},\dunderline{X}^{0,\emptyset}_{t} )$ (of size $y-x$) equals $k$ with probability $e^{-b(y-x)}(b(y-x))^k/k!$. 
Conditionally on $K=k$ (recall that $k$ can be zero), the heights $a+z_1\leq \cdots\leq a+z_k$ at which they are grafted are the order statistics of $k$ iid uniform random variables on $(0,y-x)$, hence have density $k!/(y-x)^k$ on the adequate simplex. 
Then, when passing to the genealogy, what gets grafted to the infinite branch point are iid processes with laws $\gamma_x^\#$ (to the left) and, in alternating fashion, $\gamma^{\#}_{z_k-z_{k-1}}, \gamma^{\text{lc}},\ldots, \gamma^{\#}_{z_2-z_1}, \gamma^{\text{lc}}$ and $\gamma^\#_{z_1-z_0}$. 
Using the description of the jumps and overshoots of $X^{\emptyset,0}$ above its cumulative infimum process of Proposition \ref{characterizationOfExcursionsAboveFutureMinimumProposition}, 
we see that infinite branch points get grafted along the leftmost infinite line of descent in $\gamma^{\text{lc}}$ as a Poisson point process with the intensity $\nu^d$ of the statement. 

On the other hand, the $(S_i)$ that get attached along the infinite line of descent of $S_\emptyset$, not on a gap but at a height of the form $\dunderline{X}^{0,\emptyset}_t$ for some $t$, when passing to the genealogy, corresponds to a tree with law $\gamma^{\text{lc}}$ that gets attached at height $\dunderline{L}_t$. 
As in Proposition \ref{PoissonDescriptionForGenealogyOfSINTreeProposition}, we see that the trees $G^*_1,G^*_2,\ldots$  not grafted at infinite branch points get grafted as a Poisson point process along the leftmost infinite line of descent of $G_\emptyset$ at rate $\beta b$. 
Together with the Poisson description of $\Upsilon$, we deduce our statement. 
\end{proof}


Armed with our Poisson description of $\gamma^{\text{lc}}$ we can give a proof of the Ray-Knight theorem for this measure. 
As before, we suppose that $\Gamma=\paren{\paren{\tau,d,\rho},\leq,\rho}$ has measure $\gamma^{\text{lc}}$ and set $\imf{\delta}{\sigma}=\imf{d}{\sigma,\rho}$ for any $\sigma\in\tau$. 
Recall that the pair $(Z^1,Z^2)$ is defined by letting $Z^1_t$ be the quantity of prolific individuals at distance $t$ from the root of our supercritical L\'evy tree $\Gamma$,  and that $Z^2$ is the density of $\mu\circ\delta^{-1}$ with respect to Lebesgue measure. 
\begin{proof}[Proof of Theorem \ref{RayKnightTheorem}]

Let us turn to the analysis of the bivariate process $(Z^1,Z^2)$ under $\gamma^{\text{lc}}$. 

We first describe the semigroup $\imf{P_t}{(n,z),\cdot}$ that will be relevant. 
For $(n,z)$ let $Z^z$ be a $\imf{\cb}{\Psi^\#}$ process starting at $z$ and, for $i$ between $1$ and $n$, let $(Z^{i,1}, Z^{i,2})$ have the same law as $(Z^1,Z^2)$. 
Furthermore, assume independence for these $n+1$ processes. 
Now, define 
\begin{lesn}
\imf{P_t}{(n,z),\cdot}\text{ as the law of }(Z^{1,1}_t+\cdots +Z^{n,1}_t,Z^z_t+Z^{1,2}_t+\cdots+Z^{n,2}_t). 
\end{lesn}Because of the branching property of  $\imf{\cb}{\Psi^\#}$, we see that $(P_t,t\geq 0)$ has the following branching property:
\begin{lesn}
\text{ the convolution }\imf{P_t}{(n_1,z_1),\cdot}*\imf{P_t}{(n_2,z_2),\cdot}\text{ equals } \imf{P_t}{(n_1+n_2,z_1+z_2),\cdot}. 
\end{lesn}Hence, to prove both that $(Z^1,Z^2)$ is a two-type branching process and that $(P_t)$ is a semigroup, it suffices to prove that $(Z^1,Z^2)$ is Markovian with transition kernels $(P_t)$.

To prove that $(Z^1,Z^2)$ is Markovian, we will add the results obtained on each infinite line of descent. 
Suppose that $Z^1_t=n$, so that there are $n$ infinite lines of descent intersecting height $t$. 
Recall that as a consequence of the spine decomposition of CBI, 
the contribution after $t$ of each spine naturally decomposes as the contribution of the trees that attach below $t$ 
(evolving as a $\imf{\cb}{\Psi^\#}$), 
and the contribution from each spine above $t$, evolving as a $\imf{\cbi}{\Psi^\#,\Phi}$. 
Recall that spines are independent. 
Thanks to the branching property, the first contribution then evolves as a $\imf{\cb}{\Psi^\#}$ started at $Z^2_t$, while the second contribution evolves as a $\imf{\cbi}{\Psi^\#,n\Phi}$. 
Their sum is therefore independent of $Z^1$ and $Z^2$ on $[0,t]$ given $(Z^1_t,Z^2_t)$ and evolves using the transition kernels  $P$ we have just described. 
We conclude that $(Z^1,Z^2)$ is a two-type branching process, where $Z^1$ is piecewise constant and non-decreasing. 
The same argument proves that $Z^1$ is a branching process all by itself  whose jump rates are determined by the Poisson description of Proposition \ref{PoissonDescriptionSuperCriticalLevyTreeProposition} and equal those in the statement of Theorem \ref{RayKnightTheorem}. 

We now compute the infinitesimal generator of $(Z^1,Z^2)$. 
For this, we decompose at the first jump $T$ of $Z^1$. 
Suppose that $\Delta Z^1_T=n$, so that we get $n$ additional infinite lines of descent; 
we additionally obtain some compact trees, which thanks to the branching property, make a jump of $Z^2$. 
When $Z^1=1$, the Poisson description of Proposition \ref{PoissonDescriptionSuperCriticalLevyTreeProposition} tells us that 
a jump of $(Z^1,Z^2)$ of size in $\set{n}\times A$ arrives at rate 
\begin{lesn}
\beta b\indi{n=1}+\int_A \frac{b^n x^{n+1}}{(n+1)!}\, \imf{\pi}{dx}. 
\end{lesn}Since on any interval on which $Z^1$ equals $n$, $Z^2$ behaves as a $\imf{\cbi}{\Psi^\#,n\Phi}$, we see that if $\imf{f}{n,z}=s^ne^{-\lambda z}$ then
\begin{linenomath}
\begin{align*}
\left.\frac{d}{dt}\right|_{t=0} \imf{P_t f}{n,z}
&= e^{-\lambda z}s^n \bra{z \imf{\Psi^\#}{\lambda}-n\imf{\Phi}{\lambda}}
\\&+ e^{-\lambda z}s^n n\bra{s-1}\beta b
\\&+\sum_{n=1}^\infty \int_0^\infty \bra{\imf{f}{n+\tilde n,z+\tilde z}-\imf{f}{n,z} }n \frac{b^n x^{n+1}}{(n+1)!}\, \imf{\pi}{dx}. 
\end{align*}
\end{linenomath}The geometric series and algebraic manipulations (based on the equality $\imf{\Psi}{b}=0$ and the definition of $\Phi$) 
then let us write the above as
\begin{lesn}
\beta bs(s-1)-2\beta\lambda s+\frac{1}{b}\int_0^\infty e^{-(\lambda+b(1-s))x} -e^{-(\lambda+b)x}+se^{-b x} -s\, \imf{\pi}{dx}.
\end{lesn}

In \cite{MR2455180}, the two-type branching process with values on $\na\times [0,\infty)$ has a semigroup characterized by
\begin{lesn}
\imf{\tilde P_tf}{n,z}=
e^{-z \bra{\imf{u_t}{\lambda+b}-b}} \bra{\frac{1}{b}\bra{\imf{u_t}{\lambda+b}-\imf{u_t}{\lambda+b(1-s)}}}^n. 
\end{lesn}where the function $u_t$ satisfies
\begin{lesn}
\imf{u_t}{\lambda}=\lambda-\int_0^t \imf{\Psi}{\imf{u_t}{\lambda}}. 
\end{lesn}We then observe that the infinitesimal generator of $\tilde P_t$ satisfies:
\begin{lesn}
\left. \frac{d}{dt}\imf{\tilde P_tf}{n,z}\right|_{t=0}=e^{-\lambda x}s^n\bra{z\imf{\Psi}{\lambda +b}}
+e^{-\lambda x}ns^{n-1}\frac{1}{b}\bra{\imf{\Psi}{\lambda +b(1-s)}-\imf{\Psi}{\lambda+b}}. 
\end{lesn}Again, algebraic manipulations show us that the generators of $P_t$ and $\tilde P_t$ at $f$ are the same. 
By the monotone class theorem (or its functional version as in \cite[Thm. 0.2.4]{MR1725357})
we conclude that $P_t$ and $\tilde P_t$ coincide. 
\end{proof}

\bibliography{GenBib}
\bibliographystyle{amsalpha}
\end{document}